\documentclass[12pt,a4paper]{article}
\usepackage{bbm}
\usepackage{stmaryrd}
\usepackage{}
\usepackage{bbding}
\usepackage{mathrsfs}
\usepackage{amssymb}
\usepackage{amsfonts}
\usepackage{amsmath}
\usepackage{cases}
\usepackage{amsthm}
\usepackage[square, comma, sort&compress, numbers]{natbib}

\usepackage[top=1in,bottom=1in,left=1.20in,right=1.20in]{geometry}

\newtheorem{theorem}{Theorem}[section]
\newtheorem{lemma}{Lemma}[section]
\newtheorem{definition}{Definition}[section]

\newtheorem{remark}{Remark}[section]

\numberwithin{equation}{section}
\theoremstyle{plain}

\newcommand{\me}{\mathrm{e}}
\newcommand{\mi}{\mathrm{i}}

\makeatletter

\newcommand{\Rmnum}[1]{\expandafter\@slowromancap\romannumeral #1@}
\makeatother

\DeclareMathOperator{\meas}{meas}

\begin{document}
\title{A KAM Theorem for Higher Dimensional  Reversible Nonlinear Schr\"{o}dinger Equations}
\author{Yingnan Sun\footnote{E-mail: flzssun@qq.com}, Zhaowei Lou\footnote{E-mail: zhwlou@nju.edu.cn},  Jiansheng Geng\footnote{Corresponding author. E-mail: jgeng@nju.edu.cn}\\
\\\\Department of  Mathematics, Nanjing University \\
Nanjing 210093, P.R. China}
\date{}
\maketitle
\begin{abstract}
In the  paper,  we prove an abstract KAM (Kolmogorov-Arnold-Moser) theorem for infinite dimensional
reversible systems.  Using this KAM theorem, we obtain
 the existence and linear stability of
 quasi-periodic solutions for a class of reversible (non-Hamiltonian) coupled nonlinear Schr\"{o}dinger systems on $d-$torus $\mathbb{T}^d$.
\end{abstract}



\section[]{Introduction and Main Result}
\subsection{Introduction}

Among various techniques for studying the existence of quasi-periodic solutions of nonlinear partial differential
equations (PDEs), KAM theory is one of the most powerful  tools. Kuksin \cite{Kuksin87} and Wayne \cite{Wayne90} first developed
 Newtonian scheme to investigate quasi-periodic solutions of Hamiltonian PDEs in one dimensional space. The general idea is that Hamiltonian function is thought of as a normal form plus a real analytic perturbation, then constructing an infinite  symplectic
transformation sequences to make the perturbation smaller and smaller and construct a converged local normal form. The normal form is helpful to understand the dynamics around the quasi--periodic solutions, for example, one sees the linear stability and zero Lyapunov exponents.

The feasibility of KAM method in one dimensional Hamiltonian PDEs, however, depends crucially on the second Melnikov condition.  Due to multiple eigenvalues of linear operator, such condition is not naturally available in higher dimensional case, and KAM method is in general not easy to apply. In 1998, Bourgain \cite{Bourgain98} first made a breakthrough. He used  multi-scale analysis method to avoid the cumbersome second Melnikov condition and thus obtained small-amplitude quasi-periodic solutions of two dimensional
nonlinear Schr\"{o}dinger equations (NLS).  Later, he improved his method and studied  quasi-periodic solutions of NLS and nonlinear wave equations in
any  dimensional space. Following the idea and method in \cite{Bourgain98}, abundant works \cite{Bourgain05, Berti13b, WangW} have been done.

There are strong hopes to develop  KAM theory  for  higher dimensional PDEs, because multi-scale analysis can not help us to understand the dynamics
around quasi--periodic solutions.  Geng and You \cite{Geng06} first  built a KAM theorem for higher dimensional  beam equation and nonlocal smooth NLS. They used  momentum conservation  condition which means the nonlinearity is independent of spatial variable $x$ to avoid the difficulty of multiple eigenvalues. In 2010, Eliasson and Kuksin \cite{Eliasson10} studied a class of higher dimensional NLS with  convolution type potential and nonlinearity containing spatial variable $x$. They used the block diagonal normal form structure to
 deal with multiple eigenvalues of linear operator. Besides, they introduced Lipschitz domain property of perturbation to handle infinitely many resonances at each KAM step.  By developing  T\"{o}plitz-Lipschitz property of perturbation and  constructing appropriate tangential sites on $\mathbb{Z}^2$, Geng, Xu and You \cite{Geng11}  got the quasi-periodic solutions of two-dimensional completely resonant NLS . Later on, Geng and You \cite{Geng13}
simplified the proof of \cite{Eliasson10} via  momentum conservation  condition. C.Procesi and M.Procesi \cite{Procesi15} extended the result in \cite{Geng11} to the $d$-dimensional case by a very ingenious choice of tangential sites.  See also \cite{Eliasson16, Procesi13, Xu09, Geng06b} for further studies.

Recently, KAM theory for Hamiltonian PDEs has been generalized to reversible ones in one dimensional space \cite{Zhang11, Berti14}.
To the best of our knowledge, there is  not any KAM result for higher dimensional reversible PDEs yet.
In fact, reversible PDEs are  a class of physically important
PDEs as well as Hamiltonian ones. For example, the following coupled NLS system arising from nonlinear optics (see \cite{Newell92}):
 \begin{equation}\label{CNLS1}
                       \begin{cases}
                        \mi u_t-\Delta u+M_{\xi} u+\partial_{\bar{u}}G_1(|u|^2,|v|^2)=0,\quad                                  \\
                        \mi v_t-\Delta v+M_{\widetilde{\xi}}v+\partial_{\bar{v}}G_2(|u|^2,|v|^2)=0,\,x\in\mathbb{T}^d:=\mathbb{R}^d/{2\pi \mathbb{Z}^d},
                       \end{cases}
\end{equation}
 where  $M_{\xi}$ and  $M_{\widetilde{\xi}}$ are real Fourier multiplier,  $G_i=o(|u|^3+|v|^3),\,i=1,2$ are real analytic functions
near $(u,v)=(0,0).$
When $G_1=G_2,$  equation \eqref{CNLS1} is not only reversible (with respect to the  involution $S_0(u(x), v(x))$$=(\bar{u}(-x), \bar{v}(-x))$
) but also Hamiltonian, and  quasi-periodic solutions
for this case  were recently obtained via Hamiltonian KAM theory in \cite{Zhou}. When $G_1\neq G_2,$  equation \eqref{CNLS1}
is no longer  Hamiltonian but still reversible. This motives us  to  develop reversible KAM theory for equation \eqref{CNLS1}.

As in the Hamiltonian case, the major difficulty in constructing  KAM scheme for equation \eqref{CNLS1} is
also to deal with infinitely many resonances.
In this paper,  by introducing the class of T\"{o}plitz-Lipschitz  vector fields (inspired by \cite{Eliasson10, Geng11, Berti14}) and momentum conservation  condition, the difficulty can be overcome.
T\"{o}plitz-Lipschitz  vector field plays the most essential role and  it reduces  infinitely many resonances to only finitely many ones.
momentum conservation  condition can simplify the proof.
We mention that T\"{o}plitz-Lipschitz  vector field introduced here is  the generalization of  T\"{o}plitz-Lipschitz  functions in \cite{Geng11}.

Following \cite{Eliasson10},  we could study more general equation \eqref{CNLS1} with nonlinearities $G_i$ containing the spatial variable $x$ explicitly,
but the proof would be more complicated since we have to deal with block diagonal normal form.
Our present paper is working on NLS  with the external parameters, and the more interesting completely resonant case  (i.e. no Fourier  multiplier in equation \eqref{CNLS1}) will be in our forthcoming paper  \cite{GengLS}. As in the Hamiltonian case (see \cite{Geng11, Procesi15}), the construction of Birkhoff normal form will be  a  new challenge.

%

\subsection{Main result}

Let  $\mathcal{I}_1=\{i^{(1)}, i^{(2)},\cdots,i^{(n)}\}\subset \mathbb{Z}^d$ and  $\mathcal{I}_2=\{\tilde{i}^{(1)}, \tilde{i}^{(2)},\cdots,\tilde{i}^{(m)}\}\subset \mathbb{Z}^d$ be two sets of distinguished sites of Fourier modes. For some technical convenience, we suppose
 $0\in\mathcal{I}_1\cap \mathcal{I}_2.$
Denote
by $\lambda_i,\,i\in \mathbb{Z}^d$  ($resp.$  $\widetilde{\lambda}_i$)
 the eigenvalues of  $-\Delta+M_\xi$ ($resp.$ $-\Delta+M_{\widetilde{\xi}}$ ) under periodic boundary conditions:
$$\lambda_{i^{(j)}}=\omega_j=|i^{(j)}|^2+\xi_j,\,\,1\leq j\leq n,$$
$$\lambda_{i}=|i|^2,\,\,i\notin \mathcal{I}_1,$$
$$\widetilde{\lambda}_{\tilde{i}^{(j)}}=\widetilde{\omega}_j=|\tilde{i}^{(j)}|^2+\widetilde{\xi}_j,\,\,1\leq j\leq m,$$
$$\widetilde{\lambda}_{i}=|i|^2,\,\,i\notin \mathcal{I}_2,$$
and the corresponding eigenfunctions $\phi_i(x)=\frac{1}{\sqrt{(2\pi)^{d}}}\me^{\mi \langle i,x\rangle}.$
Assume the parameters
$(\xi, \tilde{\xi})\in \mathcal{O}:=[0,1]^n\times [0,1]^m\subset \mathbb{R}^{n+m}.$

Then we have the following main result.
\begin{theorem}\label{mainresult}
For any $0<\gamma\ll1,$
there exists a Cantor subset $\mathcal{O}_\gamma \subset \mathcal{O}$  with $\meas(\mathcal{O}\setminus \mathcal{O}_\gamma)=O(\gamma^{\frac 14})$,
such that for any $(\xi, \tilde{\xi})\in \mathcal{O}_\gamma ,$ equation \eqref{CNLS1} with reversible perturbation $G_1\neq G_2$ possesses a  small amplitude quasi-periodic solution of the form
 \begin{equation}\label{sol}
                       \begin{cases}
                       u(t,x)=\sum\limits_{i\in \mathbb{Z}^d}u_i(\acute{\omega}_1t,\cdots,\acute{\omega}_nt)\phi_i(x),\quad                                  \\
                        v(t,x)=\sum\limits_{i\in \mathbb{Z}^d}v_i(\acute{\tilde{\omega}}_1t,\cdots,\acute{\tilde{\omega}}_mt)\phi_i(x),
                       \end{cases}
\end{equation}
where $u_i:\mathbb{T}^n\rightarrow\mathbb{R}$ \emph{(}resp. $v_i:\mathbb{T}^m\rightarrow\mathbb{R}$ \emph{)}
and $\acute{\omega}_1,\cdots,\acute{\omega}_n$  \emph{(}resp. $\acute{\tilde{\omega}}_1,\cdots,\acute{\tilde{\omega}}_m$ \emph{)} are close to the unperturbed frequencies $\omega_1,\cdots,\omega_n$
\emph{(}resp. $\tilde{\omega}_1,\cdots,\tilde{\omega}_m$ \emph{)}.
Moreover,  the quasi-periodic solutions are real analytic and linearly stable.

\end{theorem}

The rest of the paper is organized as follows.
In Section 2, we give the definitions of weighted norms for functions and vector fields.
An  abstract   KAM theorem (Theorem \ref{KAM}) for infinite dimensional  reversible systems is presented in Section 3.
In Section 4, we use the KAM theorem to prove Theorem \ref{mainresult}.
The proof of  Theorem \ref{KAM} is given in  Section 5.
Some properties of reversible system and technical lemmas are listed in the Appendix.

\section[]{Preliminary}

For the sake of completeness, we first introduce some definitions and notations.

 Let $\mathcal{I}\subset \mathbb{Z}^d$ be a  finite subset  and $\rho> 0$, we introduce the Banach space $\ell^{\rho}_{\mathcal{I}}$ of all complex sequences
$z=(z_j)_{j\in \mathbb{Z}^d\setminus \mathcal{I}}$ with
$$\|z\|_\rho=\sum\limits_{j\in\mathbb{Z}^d\setminus \mathcal{I}}\me^{|j|\rho}|z_j|<\infty,$$
where $|j|=\sqrt{|j_1|^2+\cdots+|j_d|^2}$.

Given two finite subsets of $\mathbb{Z}^d:$ $\mathcal{I}_1=\{i^{(1)}, i^{(2)},\cdots,i^{(n)}\}$ and  $\mathcal{I}_2=\{\tilde{i}^{(1)}, \tilde{i}^{(2)},\cdots,\tilde{i}^{(m)}\}$.  Denote $\mathbb{Z}^d_l:=\mathbb{Z}^d\setminus \mathcal{I}_l$ and $\ell^{\rho}_{l}:=\ell^{\rho}_{\mathcal{I}_l},\,(l=1,2).$
Consider the  phase space $$\mathscr{P}_{\rho}:=\mathbb{T}^n\times\mathbb{T}^{m}\times\mathbb{R}^n\times\mathbb{R}^{m} \times \ell^{\rho}_{1}\times \ell^{\rho}_{2}\times \ell^{\rho}_{1}\times \ell^{\rho}_{2}\ni y:=(\theta, \varphi, I, J, z, w, \bar{z}, \bar{w} ).$$
We introduce a complex neighborhood
\begin{equation*}
  \begin{split}
  D_\rho(r,s)=&\{y:|\textmd{Im} \theta|<r, |\textmd{Im} \varphi|<r, |I|<s, |J|<s,\\
   &\|z\|_{\rho}<s, \|w\|_{\rho}<s,\|\bar{z}\|_{\rho}<s,\|\bar{w}\|_{\rho}<s\}
 \end{split}
\end{equation*}
of $\mathcal{T}^{n+m}_0:=\mathbb{T}^{n}\times \mathbb{T}^{m}\times \{I=0\}\times\{J=0\}\times\{z=0\}\times\{w=0\}\times\{\bar{z}=0\}\times\{\bar{w}=0\},$
where $|\cdot|$ is the sup-norm for vectors.

Suppose $\mathcal{O} \subset \mathbb{R}^{n+m}$ is a compact parameter subset. A function $f:D_\rho(r,s)\times \mathcal{O} \rightarrow \mathbb{C}$   is real analytic in
$y$ and $C^4_W$ (i.e., $C^4-$smooth in the sense of Whitney) in $\zeta\in \mathcal{O}$  and has
Taylor-Fourier series expansion
\begin{equation*}
  \begin{split}
        f(y ; \zeta)
            =&\sum\limits_{k\in \mathbb{Z}^n, l\in \mathbb{N}^n,\alpha, \beta\in \mathbb{N}^{\mathbb{Z}^d_1},\atop \tilde{k}\in \mathbb{Z}^m, \tilde{l}\in \mathbb{N}^m,\tilde{\alpha}, \tilde{\beta}\in \mathbb{N}^{\mathbb{Z}^d_2}}f_{kl\alpha \beta,\tilde{k}\tilde{l}\tilde{\alpha}\tilde{\beta}}(\zeta)\me^{\mi (\langle k, \theta\rangle+\langle \tilde{k}, \varphi\rangle)}I^lJ^{\tilde{l}} z^{\alpha}\bar{z}^{\beta}w^{\tilde{\alpha}}\bar{w}^{\tilde{\beta}} ,
   \end{split}
\end{equation*}
where $\langle k,\theta\rangle=\sum\limits^n_{i=1}k_i\theta_i,$ $I^{l}=\prod^n\limits_{i=1}I^{l_i}_i$ and $z^{\alpha}\bar{z}^{\beta}=\prod\limits_{j\in\mathbb{Z}^d_1}z^{\alpha_j}_j\bar{z}^{\beta_j}_j.$  $\alpha, \beta$ have only finitely many nonzero components, and
similarly for the other indexes.
 We define the weighted norm of $f$ as follows
\begin{equation*}
  \begin{split}
        \|f\|_{D_\rho(r,s)\times\mathcal {O}}
            =&\sup\limits_{|z|_\rho<s,|\bar{z}|_\rho<s\atop |w|_\rho<s,|\bar{w}|_\rho<s} \sum\limits_{k,l,\alpha, \beta,\atop \tilde{k},\tilde{l},\tilde{\alpha},\tilde{\beta}}|f_{kl\alpha \beta,\tilde{k}\tilde{l}\tilde{\alpha}\tilde{\beta}}|_{\mathcal{O}}\me^{(|k|+|\tilde{k}|)r}s^{(|l|+|\tilde{l}|)}|z^{\alpha}||\bar{z}^{\beta}||w^{\tilde{\alpha}}||\bar{w}^{\tilde{\beta}}|.
   \end{split}
\end{equation*}
where $|f_{kl\alpha \beta,\tilde{k}\tilde{l}\tilde{\alpha}\tilde{\beta}}|_{\mathcal{O}}=\sup\limits_{\zeta\in \mathcal{O}}\sum\limits_{0\leq b\leq4}|\partial^b_\zeta f_{kl\alpha \beta,\tilde{k}\tilde{l}\tilde{\alpha}\tilde{\beta}}|.$

Let
\begin{equation*}
z^\varrho_j=
         \begin{cases}
           z_j,\quad                                      &  \varrho=+,\\
          \bar{z}_j,\quad           &   \varrho=- ,
         \end{cases}
\end{equation*}
and similarly for $z^\varrho=(z^\varrho_j)_{j\in \mathbb{Z}^d_1},$ $w^\varrho_j$ and $w^\varrho$.

Consider a vector field $X(y), y\in D_\rho(r,s):$
\begin{equation}\label{pdoperator}
  \begin{split}
        X(y)=&X^{(\theta)}(y)\frac{\partial}{\partial \theta}+X^{(\varphi)}(y)\frac{\partial}{\partial \varphi}+X^{(I)}(y)\frac{\partial}{\partial I}+X^{(J)}(y)\frac{\partial}{\partial J}\\
        &+ X^{(z)}(y)\frac{\partial}{\partial z}+ X^{(w)}(y)\frac{\partial}{\partial w}+X^{(\bar{z})}(y)\frac{\partial}{\partial \bar{z}}+X^{(\bar{w})}(y)\frac{\partial}{\partial \bar{w}}\\
            =&\sum\limits_{\textsf{v}\in \mathscr{V}}X^{(\textsf{v})}(y)\frac{\partial}{\partial \textsf{v}},
   \end{split}
\end{equation}
where
 $\mathscr{V}=\{\theta_a, \varphi_b,\,z_i, w_j,\bar{z}_i,\bar{w}_j: a=1,\cdots,n;b=1,\cdots,m;i\in\mathbb{Z}^d_1;j\in\mathbb{Z}^d_2\}.$
Suppose $X$ is  real analytic  in $y$ and depends $C^4_W$ smoothly on parameters $\zeta\in \mathcal {O},$
we define the weighted norm of $X$ as follows\footnote{The norm of vector valued function $G:D_\rho(r,s)\times\mathcal {O}\rightarrow \mathbb{C}^n$, $n<\infty,$ is defined as $\|G\|_{D_\rho(r,s)\times\mathcal {O}}=\sum\limits^n_{b=1}\|G_b\|_{D_\rho(r,s)\times\mathcal {O}}$.}
\begin{equation*}
  \begin{split}
\|X\|_{s;D_\rho(r,s)\times\mathcal {O}}=&\|X^{(\theta)}\|_{D_\rho(r,s)\times\mathcal {O}}+\|X^{(\varphi)}\|_{D_\rho(r,s)\times\mathcal {O}}\\
&+\frac{1}{s}\|X^{(I)}\|_{D_\rho(r,s)\times\mathcal {O}}+\frac{1}{s}\|X^{(J)}\|_{D_\rho(r,s)\times\mathcal {O}}\\
&+\frac{1}{s}\sum\limits_{\sigma=\pm}(\sum\limits_{i\in \mathbb{Z}^d_1}\me^{|i|\rho}\|X^{(z^\sigma_i)}\|_{D_\rho(r,s)\times\mathcal {O}}+\sum\limits_{i\in \mathbb{Z}^d_2}\me^{|i|\rho}\|X^{(w^\sigma_i)}\|_{D_\rho(r,s)\times\mathcal {O}}).
 \end{split}
\end{equation*}

The Lie bracket of two vector fields $X$ and $Y$ is defined as
  $[X,Y]=Y X-X Y.$

\section[]{A  KAM Theorem for Infinite Dimensional Reversible Systems}

In this section, we give an abstract KAM theorem for infinite dimensional reversible systems.
The definition and properties of reversible system are listed in the Appendix.

Given an involution $S:(\theta,\varphi,I, J, z,w,\bar{z},\bar{w})\mapsto(-\theta,-\varphi,I, J, \bar{z},\bar{w},z,w).$
We consider a family of  $S-$reversible vector fields
 $$X^0(y;\zeta)=N(y;\zeta)+\mathcal{A}(y;\zeta)$$
\begin{equation}\label{VFN}
\begin{split}
  N=&\omega(\zeta)\frac{\partial}{\partial \theta}+\tilde{\omega}(\zeta)\frac{\partial}{\partial \varphi}+\mi\Omega(\zeta) z\frac{\partial}{\partial z}
  +\mi\widetilde{\Omega}(\zeta)w\frac{\partial}{\partial w}-\mi\Omega(\zeta)\bar{z}\frac{\partial}{\partial \bar{z}}-\mi\widetilde{\Omega}(\zeta)\bar{w}\frac{\partial}{\partial \bar{w}}\\
  =&\sum^{n}_{b=1}\omega_b(\zeta)\frac{\partial}{\partial \theta_b}+\sum^{m}_{b=1}\tilde{\omega}_b(\zeta)\frac{\partial}{\partial \varphi_b}\\
  &+\sum_{j\in \mathbb{Z}^d_1}(\mi\Omega_j(\zeta) z_j\frac{\partial}{\partial z_j}-\mi\Omega_j(\zeta)\bar{z}_j\frac{\partial}{\partial \bar{z}_j})
  +\sum_{j\in \mathbb{Z}^d_2}(\mi\widetilde{\Omega }_j(\zeta)w_j\frac{\partial}{\partial w_j}-\mi\widetilde{\Omega}_j(\zeta)\bar{w}_j\frac{\partial}{\partial \bar{w}_j}),
 \end{split}
\end{equation}
\begin{equation}\label{VFA}
  \begin{split}
  \mathcal{A}=&\mi A(\zeta) w\frac{\partial}{\partial z}
  +\mi\widetilde{A}(\zeta)z\frac{\partial}{\partial w}-\mi A(\zeta)\bar{w}\frac{\partial}{\partial \bar{z}}-\mi\widetilde{A}(\zeta)\bar{z}\frac{\partial}{\partial \bar{w}}\\
  =&\sum_{j\in \mathbb{Z}^d_1\cap \mathbb{Z}^d_2}(\mi A_j(\zeta)w_j\frac{\partial}{\partial z_j}+\mi \tilde{A}_j(\zeta)z_j\frac{\partial}{\partial w_j}-\mi A_j(\zeta)\bar{w}_j\frac{\partial}{\partial \bar{z}_j}-\mi \tilde{A}_j(\zeta)\bar{z}_j\frac{\partial}{\partial \bar{w}_j}),
   \end{split}
\end{equation}
where $\omega_b, \tilde{\omega}_b, \Omega_j,\,\tilde{\Omega}_j,\,A_j,\,\tilde{A}_j\in \mathbb{R}$ and $A_j=0,(j\in \mathbb{Z}^d_1\setminus \mathbb{Z}^d_2),$
$\tilde{A}_j=0,(j\in \mathbb{Z}^d_2\setminus \mathbb{Z}^d_1).$

For each $\zeta\in \mathcal{O}$, the motion equation governed by the vector field $X^0$ is
\begin{equation*}
    \begin{cases}
       \dot{\theta}=\omega,\,\,\dot{\varphi}=\tilde{\omega},\\
         \dot{I}=0,\,\,\dot{J}=0,\\
         \dot{z}^{\sigma}_j=\sigma \mi\Omega_j z^{\sigma}_j,\quad  &\sigma=\pm,\,j\in\mathbb{Z}^d_1\setminus\mathbb{Z}^d_2,\\
          \dot{w}^{\sigma}_j=\sigma \mi \tilde{\Omega}_j w^{\sigma}_j,\quad  & j\in\mathbb{Z}^d_2\setminus\mathbb{Z}^d_1,\\
          \left(
            \begin{array}{c}
              \dot{z}^{\sigma}_j \\
              \dot{w}^{\sigma}_j \\
            \end{array}
          \right)=\sigma \mi\left(
                               \begin{array}{cc}
                                 \Omega_j & A_j \\
                                 \tilde{A}_j & \tilde{\Omega}_j \\
                               \end{array}
                             \right)\left(
            \begin{array}{c}
              z^{\sigma}_j \\
              w^{\sigma}_j \\
            \end{array}
          \right),\quad  & j\in\mathbb{Z}^d_1\cap\mathbb{Z}^d_2.\\
\end{cases}
\end{equation*}
Obviously, $\{(\theta+\omega t, \varphi+\tilde{\omega} t, 0,0,0,0,0,0,):t\in \mathbb{R}\}$ forms an invariant torus of the above system.

Consider now the perturbed $S-$reversible vector field
\begin{equation}\label{perbVF}
    X=X^0+P=N+\mathcal{A}+P(y;\zeta).
\end{equation}
We will prove that, for typical  (in the sense of Lebesgue measure) $\zeta\in \mathcal{O}$, the vector fields \eqref{perbVF} still admit
 invariant tori for sufficiently small $P.$   For this purpose, we need the following six assumptions:
 \begin{description}
   \item[(A1) Non-degeneracy: ]  The map $\zeta\mapsto(\omega(\zeta), \tilde{\omega}(\zeta))$ is a $C^4_W$ diffeomorphism between $\mathcal{O}$ and its image.
   \item[(A2) Asymptotics of normal frequencies:]
   \begin{equation}\label{AsyNF}
   \Omega_j=|j|^2+\Omega^0_j,\,\,j\in \mathbb{Z}^d_1,\,\,\tilde{\Omega}_j=|j|^2+\tilde{\Omega}^0_j,\,\,j\in \mathbb{Z}^d_2,
   \end{equation}
where    $\Omega^0_j,\,\tilde{\Omega}^0_j\in C^4_W(\mathcal{O})$ with $C^4_W-$norm bounded by some small
positive constant $L.$
 \item[(A3) Non-resonance conditions: ]
Denote
\begin{equation*}
    M_j\:=\left(
          \begin{array}{cc}
            \Omega_j & A_j \\
            \tilde{A}_j & \tilde{\Omega}_j \\
          \end{array}
        \right),\,j\in\mathbb{Z}^d_1\cap\mathbb{Z}^d_2.
\end{equation*}

Suppose $ A_j,  \tilde{A}_j\in C^4_W(\mathcal{O})$ and there exist $\gamma, \tau>0, $ such that
\begin{equation}\label{Mel1}
   |\langle k,\omega\rangle+\langle \tilde{k},\tilde{\omega}\rangle|\geq\frac{\gamma}{(|k|+|\tilde{k}|)^\tau},\,\, (k,\tilde{k})\neq 0,
\end{equation}
\begin{equation}\label{Mel2}
  |\langle k,\omega\rangle+\langle \tilde{k},\tilde{\omega}\rangle + \Omega_i|\geq\frac{\gamma}{(|k|+|\tilde{k}|)^\tau},i\in \mathbb{Z}^d_1\setminus\mathbb{Z}^d_2
\end{equation}
\begin{equation}\label{Mel3}
  |\langle k,\omega\rangle+\langle \tilde{k},\tilde{\omega}\rangle+\tilde{\Omega}_i|\geq\frac{\gamma}{(|k|+|\tilde{k}|)^\tau},i\in \mathbb{Z}^d_2\setminus\mathbb{Z}^d_1,
\end{equation}
\begin{equation}\label{Mel4}
  |\langle k,\omega\rangle+\langle \tilde{k},\tilde{\omega}\rangle+ \Omega_i\pm \Omega_j|\geq\frac{\gamma}{(|k|+|\tilde{k}|)^\tau}, (k,\tilde{k})\neq 0,i,j\in \mathbb{Z}^d_1\setminus\mathbb{Z}^d_2,
\end{equation}
\begin{equation}\label{Mel5}
  |\langle k,\omega\rangle+\langle \tilde{k},\tilde{\omega}\rangle+ \Omega_i\pm \tilde{\Omega}_j|\geq\frac{\gamma}{(|k|+|\tilde{k}|)^\tau},i\in \mathbb{Z}^d_1\setminus\mathbb{Z}^d_2 ,j\in \mathbb{Z}^d_2\setminus\mathbb{Z}^d_1,
\end{equation}
\begin{equation}\label{Mel6}
  |\langle k,\omega\rangle+\langle \tilde{k},\tilde{\omega}\rangle+ \tilde{\Omega}_i\pm \tilde{\Omega}_j|\geq\frac{\gamma}{(|k|+|\tilde{k}|)^\tau}, (k,\tilde{k})\neq 0,i,j\in \mathbb{Z}^d_2\setminus\mathbb{Z}^d_1,
\end{equation}
\begin{equation}\label{Mel7}
  |\det((\langle k,\omega\rangle+\langle \tilde{k},\tilde{\omega}\rangle)I_2\pm M_i)|\geq\frac{\gamma}{(|k|+|\tilde{k}|)^\tau},i\in \mathbb{Z}^d_1\cap\mathbb{Z}^d_2,
\end{equation}
\begin{equation}\label{Mel11}
  |\det((\langle k,\omega\rangle+\langle \tilde{k},\tilde{\omega}\rangle+ \Omega_i)I_2\pm M_j)|\geq\frac{\gamma}{(|k|+|\tilde{k}|)^\tau},
\end{equation}
$(i,j)\ \hbox{or}\ (j,i)\in (\mathbb{Z}^d_1\cap\mathbb{Z}^d_2)\times (\mathbb{Z}^d_1\setminus\mathbb{Z}^d_2),$
\begin{equation}\label{Mel12}
  |\det((\langle k,\omega\rangle+\langle \tilde{k},\tilde{\omega}\rangle+ \tilde{\Omega}_i)I_2\pm M_j)|\geq\frac{\gamma}{(|k|+|\tilde{k}|)^\tau},
\end{equation}
$(i,j)\ \hbox{or}\ (j,i)\in (\mathbb{Z}^d_1\cap\mathbb{Z}^d_2)\times (\mathbb{Z}^d_2\setminus\mathbb{Z}^d_1),$
\begin{equation}\label{Mel13}
  |\det((\langle k,\omega\rangle+\langle \tilde{k},\tilde{\omega}\rangle)I_{4}+ M_i\otimes I_2\pm I_2\otimes M^T_j)|\geq\frac{\gamma}{(|k|+|\tilde{k}|)^\tau},\,\, (k,\tilde{k})\neq 0,i,j\in \mathbb{Z}^d_1\cap\mathbb{Z}^d_2.
\end{equation}
where $I_b$ is $b\times b$ identity matrix. $\det(\cdot)$, $\otimes $  and $(\cdot)^T$ denotes the determinant, the tensor product and the  transpose of matrices, respectively.

 \item[(A4) Regularity: ] $\mathcal{A}+P$ is real analytic in $y$ and $C^4_W-$smooth in $\zeta$. Moreover,
 $\|\mathcal{A}\|_{s;D(r,s)\times\mathcal{O}}<1,$ $\varepsilon_0:=\|P\|_{s;D(r,s)\times\mathcal{O}}<\infty.$

 \item[(A5) Momentum conservation: ]
 The perturbation $P$ satisfies $[P, \mathbb{M}_l]=0,\,\,(l=1,\cdots,d),$ where
\begin{equation*}
    \mathbb{M}_l=\sum^n_{b=1}i^{(b)}_l\frac{\partial}{\partial \theta_b}+\sum^{m}_{b=1}\tilde{i}^{(b)}_l\frac{\partial}{\partial \varphi_b}+ \sum_{\varrho=\pm} \sum_{j\in \mathbb{Z}^d_1}\varrho \mi j_lz^{\varrho}_j\frac{\partial}{\partial z^{\varrho}_{j}}+ \sum_{\varrho=\pm} \sum_{j\in \mathbb{Z}^d_2}\varrho \mi j_lw^{\varrho}_j\frac{\partial}{\partial w^{\varrho}_{j}}.
\end{equation*}

\item[(A6) T\"{o}plitz-Lipschitz property: ]
Let
$$\Lambda=\sum\limits_{\sigma=\pm}\sigma\mi(\sum\limits_{j\in \mathbb{Z}^d_1}\Omega^0_j(\zeta) z^\sigma_j\frac{\partial}{\partial z^\sigma_j}
  +\sum\limits_{j\in \mathbb{Z}^d_2}\widetilde{\Omega }^0_j(\zeta)w^\sigma_j\frac{\partial}{\partial w^\sigma_j}).$$
For  fixed $i,j\in \mathbb{Z}^d$ $c\in \mathbb{Z}^d\setminus\{0\}, $ the following limits exist and satisfy:
\begin{equation}\label{tl0}
     \|\lim_{t\rightarrow\infty}\frac{\partial P^{(x)}}{\partial u^\sigma_{i+tc} }\|_{s;D_\rho(r,s)\times\mathcal{O}}\leq \varepsilon_0, \,\,\,\,x=\theta_b, \varphi_b, I_b,J_b, u=z, w.
\end{equation}
\begin{equation}\label{tl1}
\|\lim_{t\rightarrow\infty}\frac{\partial (\Lambda+P)^{(u^\sigma_{i+tc})}}{\partial u^{\pm\sigma}_{j\pm tc}}\|_{s;D_\rho(r,s)\times\mathcal{O}}
\leq \varepsilon_0\me^{-|i\mp j|\rho},\,\,u=z,w.
\end{equation}

\begin{equation}\label{tl2}
\|\lim_{t\rightarrow\infty}\frac{\partial (\mathcal{A}+P)^{(u^\sigma_{i+tc})}}{\partial v^{\pm\sigma}_{j\pm tc}}\|_{s;D_\rho(r,s)\times\mathcal{O}}\leq \varepsilon_0\me^{-|i\mp j|\rho},\,\,\,(u,v)=(z,w),(w,z).
\end{equation}

Furthermore, there exists $K>0$  such that when  $|t|>K,$ the following estimates hold.
\begin{equation}\label{tl3}
   \|\frac{\partial P^{(x)}}{\partial u^\sigma_{i+tc} }-\lim_{t\rightarrow\infty}\frac{\partial P^{(x)}}{\partial u^\sigma_{i+tc} }\|_{s;D_\rho(r,s)\times\mathcal{O}}\leq \varepsilon_0,
\end{equation}
$x=\theta_b, \varphi_b, I_b,J_b; u=z, w.$
\begin{equation}\label{tl4}
\|\frac{\partial (\Lambda+P)^{(u^{\sigma}_{i+tc})}}{\partial u^{\pm\sigma}_{j\pm tc}}-\lim_{t\rightarrow\infty}\frac{\partial (\Lambda+P)^{(u^{\sigma}_{i+tc})}}{\partial u^{\pm\sigma}_{j\pm tc}}\|_{s;D_\rho(r,s)\times\mathcal{O}}\leq \frac{\varepsilon_0}{|t|}\me^{-|i\mp j|\rho},
\end{equation}
$u=z,w.$
\begin{equation}\label{tl5}
\|\frac{\partial (\mathcal{A}+P)^{(u^\sigma_{i+tc})}}{\partial v^{\pm\sigma}_{j\pm tc}}-\lim_{t\rightarrow\infty}\frac{\partial (\mathcal{A}+P)^{(u^\sigma_{i+tc})}}{\partial v^{\pm\sigma}_{j\pm tc}}\|_{s;D_\rho(r,s)\times\mathcal{O}}\leq \frac{\varepsilon_0}{|t|}\me^{-|i\mp j|\rho},
\end{equation}
$(u,v)=(z,w), (w,z).$
 \end{description}

\begin{remark}
In (A6), the conditions \eqref{tl1}-\eqref{tl2} and \eqref{tl4}-\eqref{tl5} are the most important  for measure estimates. The
role played by the conditions \eqref{tl0} and \eqref{tl3}  is to preserve T\"{o}plitz-Lipschitz property   after the KAM iteration
(see Lemmas \ref{LieBTL} and \ref{PTL}).
\end{remark}

Now we state our KAM theorem.
\begin{theorem}\label{KAM}
Suppose the $S-$reversible vector field $X=N+\mathcal{A}+P$ in \eqref{perbVF} satisfies $(A1)-(A6).$
$\gamma>0$ is small enough. Then there exists a positive $\varepsilon$ depending only on $n, m, L, K, \tau, r, s$
and $\rho$ such that if $\|P\|_{s;D(r,s)\times\mathcal{O}}\leq\varepsilon,$   the following holds: There exist (1)
 a Cantor subset $\mathcal{O}_\gamma\subset\mathcal{O}$ with Lebesgue measure $\meas(\mathcal{O}\setminus\mathcal{O}_\gamma)=O(\gamma^{1/4});$
(2) a $C^4_W-$smooth family of  real analytic  torus embeddings
$$\Psi:\mathbb{T}^{n+m}\times\mathcal{O}_\gamma\rightarrow D_\rho(r,s)$$
which is $\frac{\varepsilon}{\gamma^{20}}-$close to the trivial embedding
$\Psi_0:\mathbb{T}^{n+m}\times\mathcal{O}\rightarrow \mathcal{T}^{n+m}_0;$
(3) a $C^4_W-$smooth  map $\phi:\mathcal{O}_\gamma\rightarrow \mathbb{R}^{n+m}$
which is  $\varepsilon-$close to the unperturbed frequency $(\omega, \tilde{\omega})$
such that for every $\zeta\in \mathcal{O}_\gamma$ and $(\theta,\varphi)\in\mathbb{T}^{n+m}$
the curve $t\mapsto \Psi((\theta,\varphi)+\phi(\zeta)t; \zeta)$ is a quasi-periodic solution of the equation governed by
the  vector field $X=N+\mathcal{A}+P.$
\end{theorem}

\section[]{Application to the Coupled NLS}
\subsection{Lattice form of equation \eqref{CNLS1}}

Let  $\mathcal{I}_1=\{i^{(1)}, i^{(2)},\cdots,i^{(n)}\}\subset \mathbb{Z}^d$ and  $\mathcal{I}_2=\{\tilde{i}^{(1)}, \tilde{i}^{(2)},\cdots,\tilde{i}^{(m)}\}\subset \mathbb{Z}^d$ and $0\in\mathcal{I}_1\cap \mathcal{I}_2.$
Under periodic boundary conditions, we denote   the
eigenvalues of  $-\Delta+M_\xi$ and $-\Delta+M_{\widetilde{\xi}}$ by $\lambda_i,\,i\in \mathbb{Z}^d$  and $\widetilde{\lambda}_i,\,i\in \mathbb{Z}^d$, respectively, satisfying
$$\omega_j=\lambda_{i^{(j)}}=|i^{(j)}|^2+\xi_j,\,\,1\leq j\leq n,$$
$$\Omega_i=\lambda_{i}=|i|^2,\,\,i\notin \mathcal{I}_1,$$
$$\widetilde{\omega}_j=\widetilde{\lambda}_{\tilde{i}^{(j)}}=|\tilde{i}^{(j)}|^2+\widetilde{\xi}_j,\,\,1\leq j\leq m,$$
$$\widetilde{\Omega}_i=\widetilde{\lambda}_{i}=|i|^2,\,\,i\notin \mathcal{I}_2,$$
and the corresponding eigenfunctions $\phi_i(x)=(2\pi)^{-\frac{d}{2}}\me^{\mi \langle i,x\rangle}.$

Without loss of generality, we consider the equation \eqref{CNLS1} when $G_1=|u|^4|v|^2 , G_2=|u|^2|v|^2$ since the higher order terms of nonlinearities will not
make any difference.

Let $u(t,x)=\sum\limits_{h\in \mathbb{Z}^d}q_h(t)\phi_h(x),$ $v(t,x)=\sum\limits_{h\in \mathbb{Z}^d}p_h(t)\phi_h(x),$
then we obtain the equivalent lattice reversible equations
 \begin{equation}\label{CNLS2}
                       \begin{cases}
                        \dot{q}_h=\mi\lambda_h q_h+Q^{(q_h)}(q,p,\bar{q},\bar{p}),\quad                                  \\
                        \dot{p}_h=\mi\widetilde{\lambda}_h p_h+\widetilde{Q}^{(p_h)}(q,p,\bar{q},\bar{p})\\
                        \dot{\bar{q}}_h=-\mi\lambda_h \bar{q}_h+Q^{(\bar{q}_h)}(q,p,\bar{q},\bar{p}),\quad                                  \\
                        \dot{\bar{p}}_h=-\mi\widetilde{\lambda}_h \bar{p}_h+\widetilde{Q}^{(\bar{p}_h)}(q,p,\bar{q},\bar{p}),\,h\in \mathbb{Z}^d,
                       \end{cases}
\end{equation}
which is reversible with respect to $S(q,p, \bar{q}, \bar{p})=(\bar{q}, \bar{p},q,p),$
where
\begin{equation}\label{gn1}
Q^{(q_h)}=\sum\limits_{i,j,k,l,m\in \mathbb{Z}^d}Q^{(q_h)}_{ijklm}q_iq_jp_k\bar{q}_l\bar{p}_m,\,Q^{(\bar{q}_h)}=\overline{Q^{(q_h)}},
\end{equation}
and
\begin{equation}\label{gn2}
\tilde{Q}^{(p_h)}=\sum\limits_{i,j,k\in \mathbb{Z}^d}\tilde{Q}^{(p_h)}_{ijk}q_ip_j\bar{q}_k,\,\tilde{Q}^{(\bar{p}_h)}=\overline{\tilde{Q}^{(p_h)}}
\end{equation}
with

\begin{equation}\label{gn3}
  \begin{split}
 Q^{(q_h)}_{ijklm}=&2\mi \int_{\mathbb{T}^d}\phi_i\phi_j\phi_k\bar{\phi}_l\bar{\phi}_m\bar{\phi}_h dx\\
                  =&
                       \begin{cases}
                        \frac{2\mi}{(2\pi)^{2d}},\quad                                      & i+j+k-l-m-h=0,\\
                        0,                       \quad                   &  i+j+k-l-m-h\neq 0,
                       \end{cases}
\end{split}
\end{equation}
and
\begin{equation}\label{gn4}
  \begin{split}
 \tilde{Q}^{(p_h)}_{ijk}=&\mi \int_{\mathbb{T}^d}\phi_i\phi_j\bar{\phi}_k\bar{\phi}_h dx\\
                  =&
                       \begin{cases}
                        \frac{\mi}{(2\pi)^{d}} ,\quad                                      & i+j-k-h=0,\\
                        0                         \quad                   &  i+j-k-h\neq 0,
                       \end{cases}
\end{split}
\end{equation}

By direct computation, one can verify that
the perturbations $Q^{(q)}=(Q^{(q_h)})_{h\in \mathbb{Z}^d_1}
$ and  $\tilde{Q}^{(p)}=(\tilde{Q}^{(p_h)})_{h\in \mathbb{Z}^d_2}
$ have the following regularity property.
\begin{lemma}\label{reg}
For any fixed $\rho > 0$, $Q^{(q)}$ \emph{(}resp. $\tilde{Q}^{(p)}$\emph{)} is real analytic as a map in a neighborhood
of the origin with
$$\|Q^{(q)}\|_{\rho}\leq c \|(q,p)\|^5_{\rho},\,\,(resp.\|\tilde{Q}\|^{(p)}_{\rho}\leq c \|(q,p)\|^3_{\rho}).$$
\end{lemma}

Let $P^0=\sum\limits_{\varrho=\pm}(Q^{(q^\varrho)}\frac{\partial}{\partial q^\varrho}+\tilde{Q}^{(p^\varrho)}\frac{\partial}{\partial p^\varrho}),$ then we have
the following lemma.
\begin{lemma}\label{A56}
\emph{(1)} $[P^0, \mathbb{M}^0_l]=0,\,l=1,\cdots,d,$
where
 \begin{equation}\label{M0}
    \mathbb{M}^0_l= \sum_{\varrho=\pm} \sum_{j\in \mathbb{Z}^d}\varrho \mi j_lq^{\varrho}_j\frac{\partial}{\partial q^{\varrho}_{j}}+ \sum_{\varrho=\pm} \sum_{j\in \mathbb{Z}^d}\varrho \mi j_lp^{\varrho}_j\frac{\partial}{\partial p^{\varrho}_{j}};
\end{equation}

\emph{(2)} $P^0$ satisfies T\"{o}plitz-Lipschitz property.
\end{lemma}
\begin{proof}
(1)
If we write
\begin{equation}\label{speperb}
  \begin{split}
        Q^{(q^\varrho_h)}=&\sum\limits_{\alpha, \beta, \tilde{\alpha}, \tilde{\beta}}Q^{(q^\varrho_h)}_{\alpha \beta \tilde{\alpha} \tilde{\beta}}q^\alpha \bar{q}^\beta p^{\tilde{\alpha}} \bar{p}^{\tilde{\beta}},\\
        \tilde{Q}^{(p^\varrho_h)}=&\sum\limits_{\alpha, \beta, \tilde{\alpha}, \tilde{\beta}}\tilde{Q}^{(p^\varrho_h)}_{\alpha \beta \tilde{\alpha} \tilde{\beta}}q^\alpha \bar{q}^\beta p^{\tilde{\alpha}} \bar{p}^{\tilde{\beta}},
   \end{split}
\end{equation}
then by \eqref{gn3} and \eqref{gn4}, we have  $Q^{(q^\varrho_h)}_{\alpha \beta \tilde{\alpha} \tilde{\beta}}=0$ and $\tilde{Q}^{(p^\varrho_h)}_{\alpha \beta \tilde{\alpha} \tilde{\beta}}=0$ when  $\pi_l(\alpha \beta, \tilde{\alpha} \tilde{\beta}; v)\neq 0,v= q^\varrho_h, p^\varrho_h$.
where
$$\pi_{l}(\alpha \beta, \tilde{\alpha} \tilde{\beta}; v)=\sum_{j\in \mathbb{Z}^d}(\alpha_j -\beta_j)j_l+\sum_{j\in \mathbb{Z}^d}(\tilde{\alpha}_j -\tilde{\beta}_j)j_l-\varrho h_l.$$
Note that by elementary computation, we have $[q^\alpha \bar{q}^\beta p^{\tilde{\alpha}} \bar{p}^{\tilde{\beta}}, \mathbb{M}^0_l]=\mi \pi_{l}(\alpha \beta, \tilde{\alpha} \tilde{\beta}; v)q^\alpha \bar{q}^\beta p^{\tilde{\alpha}} \bar{p}^{\tilde{\beta}},$ which implies
$[P^0, \mathbb{M}^0_l]=0.$

(2) We only consider
$\lim\limits_{t\rightarrow\infty}\frac{\partial Q^{(q_{i+tc})}}{\partial q_{j+tc}}.$
It follows from  \eqref{gn1} and \eqref{gn3} that
$$\frac{\partial Q^{(q_{i})}}{\partial q_{j}}=\sum\limits_{n+k+l-m=i-j}\frac{4\mi}{(2\pi)^{2d}}q_np_k\bar{q}_l\bar{p}_m,$$
then $\frac{\partial Q^{(q_{i+tc})}}{\partial q_{j+tc}}=\frac{\partial Q^{(q_{i})}}{\partial q_{j}}=\lim\limits_{t\rightarrow\infty}\frac{\partial Q^{(q_{i+tc})}}{\partial q_{j+tc}}.$

\end{proof}

\subsection{Verification of  assumptions (A1)-(A6)}

We introduce  action-angle variables $(\theta,  \varphi, I, J)$
and normal  coordinates $(z,w, \bar{z}, \bar{w})$ by the following  transformation $\Psi$ on some $D(r,s),$ ($r,s>0$):
\begin{equation}\label{Psi}
 \begin{split}
 &q_{i^{(j)}}=\sqrt{I_j+I^{0}_j}\me^{\mi \theta_j},\,\,\bar{q}_{i^{(j)}}=\sqrt{I_j+I^{0}_j}\me^{-\mi \theta_j},\,\,j=1,\cdots,n,\\
&p_{\tilde{i}^{(j)}}=\sqrt{J_j+J^{0}_j}\me^{\mi \varphi_j},\,\,\bar{p}_{\tilde{i}^{(j)}}=\sqrt{J_j+J^{0}_j}\me^{-\mi \varphi_j},\,\,j=1,\cdots,m,\\
&q_h=z_h,\,\,\bar{q}_h=\bar{z}_h,\,\,h\notin \mathcal{I}_1,\\
&p_h=w_h,\,\,\bar{p}_h=\bar{w}_h,\,\,h\notin \mathcal{I}_2,\\
\end{split}
\end{equation}
where the $I^{0}_j,\,\,J^{0}_j$ are  fixed numbers satisfying $0<s<I^{0}_j,\,J^{0}_j<2s.$

We  obtain a new vector field
\begin{equation}\label{resvf0}
 \begin{split}
 &X(\theta,\varphi,I, J, z,\bar{z},w,\bar{w})\\
=&N(\theta,\varphi,I, J, z,\bar{z},w,\bar{w})+P(\theta,\varphi,I, J, z,\bar{z},w,\bar{w}),
\end{split}
\end{equation}
where
\begin{equation*}
  N=\omega\frac{\partial}{\partial \theta}+\tilde{\omega}\frac{\partial}{\partial \varphi}+\mi\Omega z\frac{\partial}{\partial z}
  +\mi\widetilde{\Omega }w\frac{\partial}{\partial w}-\mi\Omega\bar{z}\frac{\partial}{\partial \bar{z}}-\mi\widetilde{\Omega}\bar{w}\frac{\partial}{\partial \bar{w}}
\end{equation*}
and
\begin{equation}\label{pyw}
  P=\sum\limits_{\textsf{w}\in\{\theta,\varphi,I, J, z,\bar{z},w,\bar{w}\}}P^{(\textsf{w})}(\theta,\varphi,I, J, z,\bar{z},w,\bar{w}; \xi,\widetilde{\xi})\frac{\partial}{\partial \textsf{w}}
\end{equation}
with
$$\omega_b=|i^{(b)}|^2+\xi_b,\,\,1\leq b\leq n,$$
$$\widetilde{\omega}_b=|\tilde{i}^{(b)}|^2+\widetilde{\xi}_b,\,\,1\leq b\leq m,$$
$$\Omega_h=|h|^2,\,\,h\notin \mathcal{I}_1,$$
$$\widetilde{\Omega}_h=|h|^2,\,\,h\notin \mathcal{I}_2.$$
\begin{equation}\label{p1}
 \begin{split}
    P^{(\theta_b)}=&\frac{1}{2\mi q_{i^{(b)}}}Q^{(q_{i^{(b)}})}\circ \Psi-
\frac{1}{2\mi \bar{q}_{i^{(b)}}}Q^{(\bar{q}_{i^{(b)}})}\circ \Psi.
 \end{split}
\end{equation}
\begin{equation}\label{p2}
     \begin{split}
    P^{(I_b)}=&\bar{q}_{i^{(b)}}Q^{(q_{i^{(b)}})}\circ \Psi+q_{i^{(b)}}Q^{(\bar{q}_{i^{(b)}})}\circ \Psi.
     \end{split}
\end{equation}
\begin{equation}\label{p3}
\begin{split}
P^{(z^\sigma_h)}=&Q^{(q^\sigma_{h})}\circ \Psi,\,\sigma=\pm.
 \end{split}
\end{equation}
\begin{equation}\label{p4}
\begin{split}
   P^{(\varphi_b)}=&\frac{1}{2\mi p_{\tilde{i}^{(b)}}}\tilde{Q}^{(p_{\tilde{i}^{(b)}})}\circ \Psi-
\frac{1}{2\mi \bar{p}_{\tilde{i}^{(b)}}}\tilde{Q}^{(\bar{p}_{\tilde{i}^{(b)}})}\circ \Psi.
\end{split}
\end{equation}
\begin{equation}\label{p5}
\begin{split}
    P^{(J_b)}=&\bar{p}_{\tilde{i}^{(b)}}\tilde{Q}^{(p_{\tilde{i}^{(b)}})}\circ \Psi+ p_{\tilde{i}^{(b)}}\tilde{Q}^{(\bar{p}_{\tilde{i}^{(b)}})}\circ \Psi.
    \end{split}
\end{equation}
\begin{equation}\label{p6}
\begin{split}
    P^{(w^\sigma_h)}=&\tilde{Q}^{(p^\sigma_{h})}\circ \Psi, \,\,\sigma=\pm.
    \end{split}
\end{equation}
$X$ is reversible with respect to the involution
$$S(\theta,\varphi, I, J, z, w, \bar{z}, \bar{w})=(-\theta,-\varphi, I, J, \bar{z}, \bar{w}, z, w).$$

Now we give the verification of  assumptions (A1)-(A6) for \eqref{resvf0}.

\emph{Verifying }(A1): Set $\zeta=(\xi, \tilde\xi)$, it is obvious as the Jacobian matrix $\frac{\partial(\omega,\tilde{\omega})}{\partial\zeta}=\frac{\partial(\omega,\tilde{\omega})}{\partial(\xi,\tilde{\xi})}=I_{n+m}.$

\emph{Verifying }(A2): It is also obvious.

\emph{Verifying }(A3): One can  refer to  Section 3.2 in \cite{Geng11} since the proof is similar.

\emph{Verifying }(A4): Suppose   vector field \eqref{resvf0} is defined
on the domain $D(r,s)$ with  $0<r<1, s=\varepsilon^{2}.$ It follows from \eqref{p1}-\eqref{p6} that
$$P^{(\theta_b)}=O(s^2),  P^{(I_b)}=O(s^3), P^{(z^\sigma_h)}=O(s^{\frac{5}{2}})\me^{-\rho|h|},$$
$$P^{(\varphi_b)}=O(s),  P^{(J_b)}=O(s^2),  P^{(w^\sigma_h)}=O(s^{\frac{3}{2}})\me^{-\rho|h|}.$$
then
$$\|P\|_{s;D_\rho(r,s)\times\mathcal{O}}\leq cs^{\frac{1}{2}}\leq c\varepsilon.$$

\emph{Verifying }(A5): Through the transformation $\Psi$ in \eqref{Psi}, the vector fields  $\mathbb{M}^0_l$
in \eqref{M0} are transformed into $\mathbb{M}_l=\Psi^*\mathbb{M}^0_l,$ then
$$[P, \mathbb{M}_l]=[\Psi^*P^0, \Psi^*\mathbb{M}^0_l]=\Psi^*[P^0, \mathbb{M}^0_l]=0.$$

\emph{Verifying }(A6): We only consider $\frac{\partial P^{(\theta_b)}}{\partial z_{j+tc}}$ and $\frac{\partial P^{(z_{i+tc})}}{\partial z_{j+tc}}$ and
the others can be verified similarly.
$$\frac{\partial P^{(\theta_b)}}{\partial z_{j+tc}}=O(s^2)\me^{-\rho|j+tc|}\rightarrow0, t\rightarrow\infty.$$
then $$\|\frac{\partial P^{(\theta_b)}}{\partial z_{j+tc}}-\lim_{t\rightarrow\infty}\frac{\partial P^{(\theta_b)}}{\partial z_{j+tc}}\|_{s;D_\rho(s,r)\times\mathcal{O}}\leq\varepsilon.$$

It follows from \eqref{p3} and  (2) in Lemma \ref{A56} that
$\frac{\partial P^{(z_{i})}}{\partial z_{j}}=O(s^\frac{5}{2})\me^{-\rho|i-j|}$
and
 $\frac{\partial P^{(z_{i+tc})}}{\partial z_{j+tc}}=\frac{\partial P^{(z_{i})}}{\partial z_{j}}=\lim\limits_{t\rightarrow\infty}\frac{\partial P^{(z_{i+tc})}}{\partial z_{j+tc}}.$
then
$$\|\frac{\partial P^{(z_{i+tc})}}{\partial z_{j+tc}}-\lim\limits_{t\rightarrow\infty}\frac{\partial P^{(z_{i+tc})}}{\partial z_{j+tc}}\|_{s;D_\rho(s,r)\times\mathcal{O}}\leq\frac{\varepsilon}{|t|}\me^{-\rho|i-j|}.$$

\section[]{Proof of Theorem \ref{KAM}}

At the $\nu$th step of  KAM iteration, we consider an $S-$reversible vector field on $D_{\rho_\nu}(r_\nu,s_\nu)\times\mathcal{O}_{\nu}:$
$$X_\nu=N_\nu+\mathcal{A}_\nu+P_\nu$$
 satisfying (A1)-(A6), where $N_\nu$ and  $\mathcal{A}_\nu$ have the same form as  $N$ and  $\mathcal{A}$ in \eqref{VFN} and \eqref{VFA}.

We shall  construct an $S-$invariant transformation $$\Phi_\nu: D_{\rho_{\nu+1}}(r_{\nu+1}, s_{\nu+1})\times\mathcal{O}_{\nu}\rightarrow D_{\rho_\nu}(r_\nu, s_\nu)\times\mathcal{O}_{\nu}$$ such that
$\Phi^*_\nu X_\nu:=(D\Phi_\nu)^{-1}\cdot X_\nu\circ \Phi_\nu=N_{\nu+1}+\mathcal{A}_{\nu+1}+P_{\nu+1}$
with new normal form $N_{\nu+1},$ $\mathcal{A}_{\nu+1}$ and a much smaller perturbation term $P_{\nu+1}$ and still satisfies (A1)-(A6).

In the sequel, for simplicity,  we drop the subscript $\nu$ and write the symbol `$+$'  for `$\nu+1$'.
Then we have vector field
  \begin{equation}\label{X}
\begin{split}
X=&N+\mathcal{A}+P\\
\end{split}
\end{equation}
with
\begin{equation*}
\begin{split}
N=&\sum^{n}_{b=1}\omega_b(\zeta)\frac{\partial}{\partial \theta_b}+\sum^{m}_{b=1}\tilde{\omega}_b(\zeta)\frac{\partial}{\partial \varphi_b}\\
  &+\sum_{\varrho=\pm}(\sum_{j\in \mathbb{Z}^d_1}\varrho\mi\Omega_j(\zeta) z^\varrho_j\frac{\partial}{\partial z^\varrho_j}
  +\sum_{j\in \mathbb{Z}^d_2}\varrho\mi\widetilde{\Omega }_j(\zeta)w^\varrho_j\frac{\partial}{\partial w^\varrho_j})
 \end{split}
\end{equation*}
\begin{equation*}
  \begin{split}
  \mathcal{A}=&\sum_{\varrho=\pm}\sum_{j\in \mathbb{Z}^d_1\cap \mathbb{Z}^d_2}(\varrho\mi A_j(\zeta)w^\varrho_j\frac{\partial}{\partial z^\varrho_j}+\varrho\mi \tilde{A}_j(\zeta)z^\varrho_j\frac{\partial}{\partial w^\varrho_j}),
   \end{split}
\end{equation*}

Let $0<r_+<r$  and
$$s_{+}=\frac{1}{4}s\varepsilon^{\frac{1}{3}},\,\varepsilon_{+}=c\gamma^{-5}(2r-2r_{+})^{-1} K^{5\tau+19}\varepsilon^{\frac{5}{3}}+\varepsilon^{\frac{7}{6}}, $$ where $c$ is some suitable  (possible different) constant independent of the iterative steps.
Then our goal  is to find a set $\mathcal{O}_+\subset\mathcal{O}$ and  an $S-$invariant transformation $\Phi: D_\rho(r_{+}, s_{+})\times\mathcal{O}_{+}\rightarrow D_\rho(r, s)\times\mathcal{O}$ such that
it transforms above
$X$ in \eqref{X}
into
 \begin{equation*}
\begin{split}
X_{+}=&N_{+}+\mathcal{A}_{+}+P_{+}\\
\end{split}
\end{equation*}
with
\begin{equation*}
\begin{split}
N_{+}=&\sum^{n}_{b=1}\omega_{+,b}(\zeta)\frac{\partial}{\partial \theta_b}+\sum^{m}_{b=1}\tilde{\omega}_{+,b}(\zeta)\frac{\partial}{\partial \varphi_b}\\
  &+\sum_{\varrho=\pm}(\sum_{j\in \mathbb{Z}^d_1}\varrho\mi\Omega_{+,j}(\zeta) z^\varrho_j\frac{\partial}{\partial z^\varrho_j}
  +\sum_{j\in \mathbb{Z}^d_2}\varrho\mi\widetilde{\Omega }_{+,j}(\zeta)w^\varrho_j\frac{\partial}{\partial w^\varrho_j})
 \end{split}
\end{equation*}
and
\begin{equation*}
  \begin{split}
  \mathcal{A}_+=&\sum_{\varrho=\pm}\sum_{j\in \mathbb{Z}^d_1\cap \mathbb{Z}^d_2}(\varrho\mi A_{+,j}(\zeta)w^\varrho_j\frac{\partial}{\partial z^\varrho_j}+\varrho\mi \tilde{A}_{+,j}(\zeta)z^\varrho_j\frac{\partial}{\partial w^\varrho_j}).
   \end{split}
\end{equation*}

\subsection{Solving the homological equations}

In the sequel,
for $K>0,$ we define the truncation operator $\mathcal{T}_K$ as follows: for  $f$  on $D(r)=\{(\theta,\varphi)\in \mathbb{C}^n\times\mathbb{C}^{m}:|\hbox{Im} \theta|<r, |\hbox{Im} \varphi|<r\},$
$$\mathcal{T}_Kf(\theta,\varphi):=\sum\limits_{(k,\tilde{k})\mathbb{Z}^n\times\mathbb{Z}^{m},\atop\ |k|+|\tilde{k}|\leq K}f_{k,\tilde{k}}\me^{\mi (\langle k, \theta\rangle+\langle \tilde{k}, \varphi\rangle)}.$$

The average of $f$ with respect to $(\theta,\varphi)$ is defined as
$$[f]=f_{0,0}=\frac{1}{(2\pi)^{n+m}}\int_{\mathbb{T}^{n+m}}f(\theta,\varphi)d\theta d\varphi$$

We write the reversible vector field $P$ as
Taylor-Fourier series
\begin{equation*}
  \begin{split}
        P(y ; \zeta)
            =&\sum\limits_{\textsf{v}\in \mathscr{V}}\sum\limits_{k, \tilde{k}, l,\tilde{l}, \alpha, \beta, \tilde{\alpha}, \tilde{\beta}}P^{(\textsf{v})}_{kl\alpha \beta,\tilde{k}\tilde{l}\tilde{\alpha}\tilde{\beta}}(\zeta)\me^{\mi (\langle k, \theta\rangle+\langle \tilde{k}, \varphi\rangle)}I^lJ^{\tilde{l}} z^{\alpha}\bar{z}^{\beta}w^{\tilde{\alpha}}\bar{w}^{\tilde{\beta}}\frac{\partial}{\partial \textsf{v}} ,
   \end{split}
\end{equation*}

Let $R=\sum\limits_{\textsf{v}\in \mathscr{V}}R^{(\textsf{v})}(y; \zeta)\frac{\partial}{\partial \textsf{v}}$
be the  truncation of  $P$, i.e.,
 for $\textsf{v}\in \{\theta_b, \varphi_b\}$,
$R^{(\textsf{v})}=\mathcal{T}_KP^{(\textsf{v})}_{000,000}(\theta,\varphi)$
and for $\textsf{v}\in \{I_b, J_b, z_j, \bar{z}_j, w_j, \bar{w}_j\}$,
\begin{equation*}
  \begin{split}
  R^{(\textsf{v})}=&\sum\limits_{ |l|+|\tilde{l}|+|\alpha|+|\beta|+|\tilde{\alpha}|+|\tilde{\beta}|\leq 1}\mathcal{T}_KP^{(\textsf{v})}_{l \alpha \beta, \tilde{l} \tilde{\alpha} \tilde{\beta}}(\theta,\varphi)I^{l}J^{\tilde{l}}z^{\alpha}\bar{z}^{\beta}w^{\tilde{\alpha}}\bar{w}^{\tilde{\beta}}\\
   \end{split}
\end{equation*}

We rewrite $R^{(\textsf{v})}$ as follows:
for $\textsf{v}\in \{\theta_b, \varphi_b\}$,
$R^{(\textsf{v})}=R^{\textsf{v}}(\theta,\varphi)$
and for $\textsf{v}\in \{I_b, J_b, z_j, \bar{z}_j, w_j, \bar{w}_j\}$,
\begin{equation*}
  \begin{split}
  R^{(\textsf{v})}=&R^{\textsf{v}}(\theta,\varphi)+\sum\limits_{\textsf{u}\in \{I_b, J_b, z_j, \bar{z}_j, w_j, \bar{w}_j\}} R^{\textsf{v}\textsf{u}}(\theta,\varphi)\textsf{u}.
   \end{split}
\end{equation*}

We define the normal form of $R$ as
\begin{equation*}
  \begin{split}
        [R] = &\sum^n_{b=1}[R^{\theta_b}]\frac{\partial}{\partial\theta_b}+\sum^{m}_{b=1}[R^{\varphi_b}]\frac{\partial}{\partial\varphi_b}\\
         &+\sum_{j\in \mathbb{Z}^d_1}[R^{z_jz_j}]z_j\frac{\partial}{\partial z_j}+\sum_{j\in \mathbb{Z}^d_1\cap \mathbb{Z}^d_2}[R^{z_jw_j}]w_j\frac{\partial}{\partial z_j}\\
         &+\sum_{j\in \mathbb{Z}^d_1}[R^{\bar{z}_j\bar{z}_j}]\bar{z}_j\frac{\partial}{\partial \bar{z}_j}+\sum_{j\in \mathbb{Z}^d_1\cap \mathbb{Z}^d_2}[R^{\bar{z}_j\bar{w}_j}]\bar{w}_j\frac{\partial}{\partial \bar{z}_j}\\
          &+\sum_{j\in \mathbb{Z}^d_2}[R^{w_jw_j}]w_j\frac{\partial}{\partial w_j}+\sum_{j\in \mathbb{Z}^d_1\cap \mathbb{Z}^d_2}[R^{w_jz_j}]z_j\frac{\partial}{\partial w_j}\\
         &+\sum_{j\in \mathbb{Z}^d_2}[R^{\bar{w}_j\bar{w}_j}]\bar{w}_j\frac{\partial}{\partial \bar{w}_j}+\sum_{j\in \mathbb{Z}^d_1\cap \mathbb{Z}^d_2}[R^{\bar{w}_j\bar{z}_j}]\bar{z}_j\frac{\partial}{\partial \bar{w}_j}\\
   \end{split}
\end{equation*}

%
%

In the sequel, denote by $\phi^t_X$  the flow generated by  vector field $X$ and $\phi^1_X=\phi^t_X|_{t=1}.$

 Suppose vector field $F$  has the same form as $R,$
$\Phi^*X=(\phi^1_F)^*X,$
\begin{equation*}
  \begin{split}
        \Phi^*X=&(\phi^1_F)^*(N+\mathcal{A})+(\phi^1_F)^*R+(\phi^1_F)^*(P-R)\\
               =&N+\mathcal{A}+[N+\mathcal{A}, F]+\int^1_0(1-t)(\phi^t_F)^*[[N+\mathcal{A},F],F]dt\\
               &+R+\int^1_0(\phi^t_F)^*[R,F]dt+(\phi^1_F)^*(P-R).
   \end{split}
\end{equation*}

We solve the  homological equation
\begin{equation}\label{HEQ}
    [N+\mathcal{A}, F]+R=[R].
\end{equation}
Denote $\partial_{(\omega,\tilde{\omega})}f(\theta, \varphi):=\partial_\omega f(\theta, \varphi)+\partial_{\tilde{\omega}}f(\theta, \varphi).$
By the definition of Lie bracket, the homological equation \eqref{HEQ} is  equivalent to the following scalar  forms $\eqref{sheq1}-\eqref{sheq4}:$
\begin{equation}\label{sheq1}
     \partial_{(\omega,\tilde{\omega})} f=g
\end{equation}
where $$(f, g)=(F^{u}, R^{u}-[R^{u}]), (F^{v}, R^{v}),$$ $u\in \{\theta_a, \varphi_a\};$ $ v\in\{I_a, J_a, I_aI_b, I_aJ_b, J_aI_b, J_aJ_b\}.$

\begin{equation}\label{sheq2}
     \partial_{(\omega,\tilde{\omega})} f- \mi \lambda f=g,
\end{equation}
where $$(f, g; \lambda)\in \{(F^{u_l}, R^{u_l}; \lambda_l):l=1,\cdots,6\}, $$
with $$u_1\in \{I_az^{-\varrho}_i, J_az^{-\varrho}_i, z^{\varrho}_i, z^{\varrho}_iI_a, z^{\varrho}_iJ_a: i\in \mathbb{Z}^d_1\setminus \mathbb{Z}^d_2\},\lambda_1=\varrho\Omega_i,$$
$$u_2\in \{I_aw^{-\varrho}_j, J_aw^{-\varrho}_j, w^{\varrho}_j, w^{\varrho}_jI_a, w^{\varrho}_jJ_a: j\in \mathbb{Z}^d_2\setminus \mathbb{Z}^d_1\},
\lambda_2=\varrho\tilde{\Omega}_j,$$
$$u_3=z^{\varrho}_iz^{\sigma}_j, \lambda_3=\varrho \Omega_i-\sigma \Omega_j, i,j\in \mathbb{Z}^d_1\setminus\mathbb{Z}^d_2,i\neq j, \varrho\neq\sigma,$$
$$u_4=z^{\varrho}_iw^{\sigma}_j, \lambda_4=\varrho \Omega_i-\sigma \tilde{\Omega}_j, i\in \mathbb{Z}^d_1\setminus\mathbb{Z}^d_2,j\in \mathbb{Z}^d_2\setminus\mathbb{Z}^d_1,$$
$$u_5=w^{\varrho}_iz^{\sigma}_j, \lambda_5=\varrho \tilde{\Omega}_i-\sigma \Omega_j, i\in \mathbb{Z}^d_2\setminus\mathbb{Z}^d_1,j\in \mathbb{Z}^d_1\setminus\mathbb{Z}^d_2,$$
$$u_6=w^{\varrho}_iw^{\sigma}_j, \lambda_6=\varrho \tilde{\Omega}_i-\sigma \tilde{\Omega}_j, i,j\in \mathbb{Z}^d_2\setminus\mathbb{Z}^d_1,i\neq j, \varrho\neq\sigma.$$
%
%
%
%
%

\begin{equation}\label{sheq3}
     \partial_{(\omega,\tilde{\omega})} \mathcal{F}+\mi \mathcal{M}\mathcal{F} =\mathcal{G},
\end{equation}
where $$( \mathcal{F}, \mathcal{G}; \mathcal{M})\in\{(\left(
                                                    \begin{array}{c}
                                                      F^{u_l} \\
                                                      F^{v_l} \\
                                                    \end{array}
                                                  \right)
, \left(
                                                    \begin{array}{c}
                                                      R^{u_l} \\
                                                      R^{v_l} \\
                                                    \end{array}
                                                  \right); \mathcal{M}_l):l=1,\cdots,6\},$$
with $$(u_1,v_1)\in \{(I_az^{\varrho}_j,I_aw^{\varrho}_j), (J_az^{\varrho}_j,J_aw^{\varrho}_j): j\in \mathbb{Z}^d_1\cap \mathbb{Z}^d_2\}, \mathcal{M}_1=\varrho M^T_j,$$
 $$(u_2,v_2)\in \{(z^{\varrho}_i,w^{\varrho}_i), (z^{\varrho}_iI_a,w^{\varrho}_iI_a), (z^{\varrho}_iJ_a,w^{\varrho}_iJ_a): i\in \mathbb{Z}^d_1\cap \mathbb{Z}^d_2\}, \mathcal{M}_2=-\varrho M_i,$$
$$(u_3,v_3)=(z^{\varrho}_iz^{\sigma}_j,z^{\varrho}_iw^{\sigma}_j), \mathcal{M}_3=-\varrho \Omega_iI_2+\sigma M^T_j,i\in \mathbb{Z}^d_1\setminus\mathbb{Z}^d_2,j\in \mathbb{Z}^d_1\cap \mathbb{Z}^d_2,$$
$$(u_4,v_4)=(w^{\varrho}_iz^{\sigma}_j,w^{\varrho}_iw^{\sigma}_j), \mathcal{M}_4=-\varrho \tilde{\Omega}_iI_2+\sigma M^T_j,i\in \mathbb{Z}^d_2\setminus\mathbb{Z}^d_1,j\in \mathbb{Z}^d_1\cap \mathbb{Z}^d_2,$$
$$(u_5,v_5)=(z^{\varrho}_iz^{\sigma}_j,w^{\varrho}_iz^{\sigma}_j), \mathcal{M}_5=\sigma \Omega_jI_2-\varrho M_i,i\in \mathbb{Z}^d_1\cap \mathbb{Z}^d_2,j\in \mathbb{Z}^d_1\setminus\mathbb{Z}^d_2,$$
$$(u_6,v_6)=(z^{\varrho}_iw^{\sigma}_j,w^{\varrho}_iw^{\sigma}_j), \mathcal{M}_6=\sigma \tilde{\Omega}_jI_2-\varrho M_i,i\in \mathbb{Z}^d_1\cap \mathbb{Z}^d_2,j\in \mathbb{Z}^d_2\setminus\mathbb{Z}^d_1.$$

\begin{equation}\label{sheq4}
\begin{cases}
\partial_{(\omega,\tilde{\omega})} F^{z^{\varrho}_iz^{\sigma}_j}-\varrho\mi \Omega_i F^{z^{\varrho}_iz^{\sigma}_j}+\sigma\mi F^{z^{\varrho}_iz^{\sigma}_j}\Omega_j-\varrho\mi A_i F^{w^{\varrho}_iz^{\sigma}_j}+\sigma\mi F^{z^{\varrho}_iw^{\sigma}_j}\tilde{A}_j=\\R^{z^{\varrho}_iz^{\sigma}_j}-\delta_{\varrho\sigma}\delta_{ij}[ R^{z^{\varrho}_iz^{\sigma}_j}],\\
\partial_{(\omega,\tilde{\omega})} F^{z^{\varrho}_iw^{\sigma}_j}-\varrho\mi \Omega_i F^{z^{\varrho}_iw^{\sigma}_j}+\sigma \mi F^{z^{\varrho}_iw^{\sigma}_j}\tilde{\Omega}_j -\varrho\mi A_i F^{w^{\varrho}_iw^{\sigma}_j}+\sigma\mi F^{z^{\varrho}_iz^{\sigma}_j}A_j=\\R^{z^{\varrho}_iw^{\sigma}_j}-\delta_{\varrho\sigma}\delta_{ij} [R^{z^{\varrho}_iw^{\sigma}_j}],\\
\partial_{(\omega,\tilde{\omega})} F^{w^{\varrho}_iz^{\sigma}_j}-\varrho\mi \tilde{\Omega}_i F^{w^{\varrho}_iz^{\sigma}_j}+\sigma\mi F^{w^{\varrho}_iz^{\sigma}_j}\Omega_j-\varrho\mi \tilde{A}_i F^{z^{\varrho}_iz^{\sigma}_j}+\sigma\mi F^{w^{\varrho}_iw^{\sigma}_j}\tilde{A}_j=\\R^{w^{\varrho}_iz^{\sigma}_j}-\delta_{\varrho\sigma}\delta_{ij}[R^{w^{\varrho}_iz^{\sigma}_j}],\\
\partial_{(\omega,\tilde{\omega})} F^{w^{\varrho}_iw^{\sigma}_j}-\varrho\mi \tilde{\Omega}_i F^{w^{\varrho}_iw^{\sigma}_j}+\sigma \mi F^{w^{\varrho}_iw^{\sigma}_j}\tilde{\Omega}_j -{\varrho}\mi \tilde{A}_i F^{z^{\varrho}_iw^{\sigma}_j}+\sigma\mi F^{w^{\varrho}_iz^{\sigma}_j}A_j=\\R^{w^{\varrho}_iw^{\sigma}_j}-\delta_{\varrho\sigma}\delta_{ij} [R^{w^{\varrho}_iw^{\sigma}_j}],
\end{cases}
\end{equation}
here $\delta_{\mu\nu}=1, \hbox{if}\ \mu=\nu,\hbox{and}\ 0, \hbox{otherwise}.$

Suppose for $\zeta\in \mathcal{O}$, $|k|+|\tilde{k}|\leq K, $
\begin{equation}\label{MelK}
  \begin{split}
   &|\langle k,\omega\rangle+\langle \tilde{k},\tilde{\omega}\rangle|\geq\frac{\gamma}{K^\tau},\,\, (k,\tilde{k})\neq 0,\\
   &|\langle k,\omega\rangle+\langle \tilde{k},\tilde{\omega}\rangle + \Omega_i|\geq\frac{\gamma}{K^\tau},i\in \mathbb{Z}^d_1\setminus\mathbb{Z}^d_2,\\
   &|\langle k,\omega\rangle+\langle \tilde{k},\tilde{\omega}\rangle+\tilde{\Omega}_i|\geq\frac{\gamma}{K^\tau},i\in \mathbb{Z}^d_2\setminus\mathbb{Z}^d_1, \\
   &|\langle k,\omega\rangle+\langle \tilde{k},\tilde{\omega}\rangle+ \Omega_i\pm \Omega_j|\geq\frac{\gamma}{K^\tau}, (k,\tilde{k})\neq 0,i,j\in \mathbb{Z}^d_1\setminus\mathbb{Z}^d_2, \\
   &|\langle k,\omega\rangle+\langle \tilde{k},\tilde{\omega}\rangle+ \Omega_i\pm \tilde{\Omega}_j|\geq\frac{\gamma}{K^\tau},i\in \mathbb{Z}^d_1\setminus\mathbb{Z}^d_2 ,j\in \mathbb{Z}^d_2\setminus\mathbb{Z}^d_1, \\
   &|\langle k,\omega\rangle+\langle \tilde{k},\tilde{\omega}\rangle+ \tilde{\Omega}_i\pm \tilde{\Omega}_j|\geq\frac{\gamma}{K^\tau}, (k,\tilde{k})\neq 0,i,j\in \mathbb{Z}^d_2\setminus\mathbb{Z}^d_1, \\
   &|\det((\langle k,\omega\rangle+\langle \tilde{k},\tilde{\omega}\rangle)I_2\pm M_i)|\geq\frac{\gamma}{K^\tau},i\in \mathbb{Z}^d_1\cap\mathbb{Z}^d_2, \\
   &|\det((\langle k,\omega\rangle+\langle \tilde{k},\tilde{\omega}\rangle+ \Omega_i)I_2\pm M_j)|\geq\frac{\gamma}{K^\tau},(i,j)\ \hbox{or}\ (j,i)\in (\mathbb{Z}^d_1\cap\mathbb{Z}^d_2)\times (\mathbb{Z}^d_1\setminus\mathbb{Z}^d_2), \\
   & |\det((\langle k,\omega\rangle+\langle \tilde{k},\tilde{\omega}\rangle+ \tilde{\Omega}_i)I_2\pm M_j)|\geq\frac{\gamma}{K^\tau},(i,j)\ \hbox{or}\ (j,i)\in (\mathbb{Z}^d_1\cap\mathbb{Z}^d_2)\times (\mathbb{Z}^d_2\setminus\mathbb{Z}^d_1), \\
   &|\det((\langle k,\omega\rangle+\langle \tilde{k},\tilde{\omega}\rangle)I_{4}+ M_i\otimes I_2\pm I_2\otimes M^T_j)|\geq\frac{\gamma}{K^\tau},\,\, (k,\tilde{k})\neq 0,i,j\in \mathbb{Z}^d_1\cap\mathbb{Z}^d_2. \\
   \end{split}
\end{equation}

\begin{lemma}\label{solF}

 Suppose that uniformly on $\mathcal{O},$ non-resonance conditions in \eqref{MelK}  hold, then
 there exist a positive $c=c(n,m,\tau)$ such that
 the  equation \eqref{HEQ}  has a unique solution $F$  with
   $[ F]=0 ,$ which
is regular on $D(r, s )\times \mathcal{O}.$
Moreover,
\begin{enumerate}
  \item
    $\|F\|_{s;D_\rho(r,s)\times \mathcal{O}}\leq c \gamma^{-5}K^{5\tau+19}\varepsilon;$
  \item $F\circ S=DS\cdot F;$
  \item $[F, \mathbb{M}_l]=0,\,\,l=1,\cdots,d;$
  \item $F$ satisfies (A6) with $\varepsilon^{\frac{2}{3}}$ in place of $\varepsilon$ on $D(r-\delta, s/2 )$, where  $0<\delta<r/2.$
\end{enumerate}

\end{lemma}

\begin{proof}
(1)
As we have mentioned above, the  equation \eqref{HEQ} is  equivalent to  $\eqref{sheq1}-\eqref{sheq4}.$
Below we only consider the most difficult equation \eqref{sheq4} with $\varrho=\sigma, i\neq j$ since the other ones can be solved similarly.

 By Fourier expansion, we have
 \begin{equation*}
((\langle k,\omega\rangle+\langle \tilde{k},\tilde{\omega}\rangle)I_{4}-\varrho M_i\otimes I_2+\varrho I_2\otimes M^T_j)
\left(
  \begin{array}{c}
    \mi F^{z^{\varrho}_iz^{\varrho}_j}_{k,\tilde{k}} \\
    \mi F^{z^{\varrho}_iw^{\varrho}_j}_{k,\tilde{k}} \\
    \mi F^{w^{\varrho}_iz^{\varrho}_j}_{k,\tilde{k}} \\
     \mi F^{w^{\varrho}_iw^{\varrho}_j}_{k,\tilde{k}} \\
  \end{array}
\right)
=\left(
  \begin{array}{c}
    R^{z^{\varrho}_iz^{\varrho}_j}_{k,\tilde{k}} \\
    R^{z^{\varrho}_iw^{\varrho}_j}_{k,\tilde{k}} \\
    R^{w^{\varrho}_iz^{\varrho}_j}_{k,\tilde{k}} \\
    R^{w^{\varrho}_iw^{\varrho}_j}_{k,\tilde{k}} \\
  \end{array}
\right)
\end{equation*}
where
\begin{equation*}
\begin{split}
&(\langle k,\omega\rangle+\langle \tilde{k},\tilde{\omega}\rangle)I_{4}-\varrho M_i\otimes I_2+\varrho I_2\otimes M^T_j\\
    =&(\langle k,\omega\rangle+\langle \tilde{k},\tilde{\omega}\rangle)I_{4}+\\
    &\left(
  \begin{array}{cccc}
    -\varrho \Omega_i +\varrho\Omega_j & \varrho \tilde{A}_j & -\varrho A_i & 0 \\
    \varrho A_j & -\varrho \Omega_i +\varrho\tilde{\Omega}_j & 0 & -\varrho A_i \\
    -\varrho \tilde{A}_i & 0 & -\varrho \tilde{\Omega}_i +\varrho\Omega_j & \varrho \tilde{A}_j \\
    0 & -\varrho \tilde{A}_i & \varrho A_j & -\varrho \tilde{\Omega}_i +\varrho\tilde{\Omega}_j \\
  \end{array}
\right).
\end{split}
\end{equation*}
Using non-resonance conditions in \eqref{MelK}, we obtain for $|k|+|\tilde{k}|\leq K,$
$$|F^{z^{\varrho}_iz^{\varrho}_j}_{k,\tilde{k}}|_{\mathcal{O}},  |F^{z^{\varrho}_iw^{\varrho}_j}_{k,\tilde{k}}|_{\mathcal{O}},
|F^{w^{\varrho}_iz^{\varrho}_j}_{k,\tilde{k}}|_{\mathcal{O}}, |F^{w^{\varrho}_iw^{\varrho}_j}_{k,\tilde{k}}|_{\mathcal{O}}$$
\begin{equation*}
\leq c\gamma^{-5}K^{5\tau+19}(|R^{z^{\varrho}_iz^{\varrho}_j}_{k,\tilde{k}}|_{\mathcal{O}}+ |R^{z^{\varrho}_iw^{\varrho}_j}_{k,\tilde{k}}|_{\mathcal{O}}+
|R^{w^{\varrho}_iz^{\varrho}_j}_{k,\tilde{k}}|_{\mathcal{O}}+|R^{w^{\varrho}_iw^{\varrho}_j}_{k,\tilde{k}}|_{\mathcal{O}}),
\end{equation*}
then according to the definition of vector fields, we have
   \begin{equation*}
    \|F\|_{s;D_\rho(r,s)\times \mathcal{O}}\leq c \gamma^{-5}K^{5\tau+19}  \|R\|_{s;D_\rho(r,s)\times \mathcal{O}}\leq c \gamma^{-5}K^{5\tau+19} \varepsilon.
\end{equation*}

(2) $F\circ S=DS\cdot F$ can be implied by the uniqueness of solutions of homological equation.

(3)
We verify $[F, \mathbb{M}_l]=0.$
Consider for $l=1,\cdots,d,$
\begin{equation}\label{piindex}
     \pi_{l}(k \alpha \beta;\tilde{k} \tilde{\alpha} \tilde{\beta};\textsf{v})=
                       \begin{cases}
                        \pi_{l}(k \alpha \beta;\tilde{k} \tilde{\alpha} \tilde{\beta}),\quad                                      &  \textsf{v}=\theta_b,\varphi_b, I_b,J_b,\\
                        \pi_{l}(k \alpha \beta;\tilde{k} \tilde{\alpha} \tilde{\beta})-\varrho j_l,\quad                   &   \textsf{v}=z^\varrho_j , w^\varrho_j,
                       \end{cases}
\end{equation}
where
$$\pi_{l}(k \alpha \beta;\tilde{k} \tilde{\alpha} \tilde{\beta})=\sum^n_{b=1}i^{(b)}_lk_l+\sum^{m}_{b=1}\tilde{i}^{(b)}_l\tilde{k}_l+\sum_{j\in \mathbb{Z}^d_1}(\alpha_j -\beta_j)j_l+\sum_{j\in \mathbb{Z}^d_2}(\tilde{\alpha}_j -\tilde{\beta}_j)j_l.$$
As in the proof of Lemma \eqref{A56}, one can verify that a vector field $X$ satisfying $[X, \mathbb{M}_l]=0$ is equivalent to
$X^{(\textsf{v})}_{k l \alpha \beta, \tilde{k} \tilde{l} \tilde{\alpha} \tilde{\beta}}=0,$ if
$\pi_{l}(k \alpha \beta;\tilde{k} \tilde{\alpha} \tilde{\beta};\textsf{v})\neq0.$
Thus in order to prove $[F, \mathbb{M}_l]=0,$  it suffices to verify that $F^{(\textsf{v})}_{k l \alpha \beta, \tilde{k} \tilde{l} \tilde{\alpha} \tilde{\beta}}=0,$ if
$\pi_{l}(k \alpha \beta;\tilde{k} \tilde{\alpha} \tilde{\beta};\textsf{v})\neq0.$
This can be implied by $[P, \mathbb{M}]=0$ since
 $F$ is determined by $R.$

 (4)  The proof can follow that of Lemma 4.3 in \cite{Geng11} since there is no essential difference.

\end{proof}

\subsection{Estimates on the coordinate transformation}

\begin{lemma}
If $\varepsilon\ll \delta\gamma^{5}K^{-5\tau-19},$ then for every $-1\leq t\leq 1,$
\begin{equation*}
    \phi^t_F:D_\rho(r-2\delta, s/4)\rightarrow D_\rho(r-\delta, s/2),
\end{equation*}
and
\begin{equation*}
    \|\phi^t_F-id\|_{s;D_\rho(r-2\delta, s/4)\times\mathcal{O}}\leq c\gamma^{-5}K^{5\tau+19}\varepsilon,
\end{equation*}
\begin{equation*}
    \|D\phi^t_F-Id\|_{s;D_\rho(r-2\delta, s/4)\times\mathcal{O}}\leq c\gamma^{-5}\delta^{-1}K^{5\tau+19}\varepsilon.
\end{equation*}
\end{lemma}
\begin{proof}
Using Cauchy's inequality, we obtain
 $$\|DF\|_{s;D_\rho(r-\delta, s/2)\times \mathcal{O}}\leq \frac{c}{\delta}\|F\|_{s;D_\rho(r,s)\times \mathcal{O}} \leq \frac{c}{\delta}\gamma^{-5}K^{5\tau+19}\varepsilon,$$ then
if $\varepsilon\ll \delta\gamma^{5}K^{-5\tau-19},$ for every $-1\leq t\leq 1,$
\begin{equation*}
    \phi^t_F:D_\rho(r-2\delta, s/4)\rightarrow D_\rho(r-\delta, s/2)
\end{equation*}
is well-defined.
Thus by Gronwall's inequality and the estimate for  $DF$, we have
   \begin{equation*}
   \begin{split}
    & \|\phi^t_{F}-id\|_{s;D(r-2\delta,s/4)\times\mathcal{O}}\\
   \leq& c\|F\|_{s;D(r,s)\times\mathcal{O}}\\
   \leq& c\gamma^{-5}K^{5\tau+19}\varepsilon
  \end{split}
  \end{equation*}
and
  \begin{equation*}
   \begin{split}
    & \|D\phi^t_{F}-Id\|_{s;D(r-2\delta,s/4)\times\mathcal{O}}\\
   \leq&c\|DF\|_{s;D(r-\delta,s/2)\times\mathcal{O}}\\
   \leq&c\delta^{-1}\gamma^{-5}K^{5\tau+19}\varepsilon.
  \end{split}
  \end{equation*}

\end{proof}

\subsection{The new normal form}
Through the time-1 map $\Phi=\phi^1_F$ defined above, vector field  $X$ is transformed into $X_+=\Phi^*X=N_+ +\mathcal{A}_++ P_+$ with
 new normal form $N_+,\mathcal{A}_+$ and  new perturbation $P_+.$

In this subsection, we consider the new normal form
 $$N_+=N+\hat{N},\,\,\,\,\mathcal{A}_+=\mathcal{A}+\hat{\mathcal{A}},$$
where
\begin{equation}\label{hatN}
  \begin{split}
       \hat{N} = &\sum^n_{b=1}[R^{\theta_b}]\frac{\partial}{\partial\theta_b}+\sum^{m}_{b=1}[R^{\varphi_b}]\frac{\partial}{\partial\varphi_b}\\
         &+\sum_{j\in \mathbb{Z}^d_1}([R^{z_jz_j}]z_j\frac{\partial}{\partial z_j}+[R^{\bar{z}_j\bar{z}_j}]\bar{z}_j\frac{\partial}{\partial \bar{z}_j})\\
          &+\sum_{j\in \mathbb{Z}^d_2}([R^{w_jw_j}]w_j\frac{\partial}{\partial w_j}+[R^{\bar{w}_j\bar{w}_j}]\bar{w}_j\frac{\partial}{\partial \bar{w}_j}),\\
   \end{split}
\end{equation}
\begin{equation*}
  \begin{split}
      \hat{\mathcal{A}} =   &\sum_{j\in \mathbb{Z}^d_1\cap \mathbb{Z}^d_2}([R^{z_jw_j}]w_j\frac{\partial}{\partial z_j}
         +[R^{\bar{z}_j\bar{w}_j}]\bar{w}_j\frac{\partial}{\partial \bar{z}_j}
          +[R^{w_jz_j}]z_j\frac{\partial}{\partial w_j}
         +[R^{\bar{w}_j\bar{z}_j}]\bar{z}_j\frac{\partial}{\partial \bar{w}_j})\\
   \end{split}
\end{equation*}
and
$$\omega_{+,b}=\omega_{b}+ [R^{\theta_b}],\,(b=1,\cdots,n),\,\,\tilde{\omega}_{+,b}=\tilde{\omega}_{b}+ [R^{\varphi_b}],\,(b=1,\cdots,m),$$
$$\Omega_{+,i}=\Omega_{i}-\mi [R^{z_iz_i}],\,(i\in \mathbb{Z}^d_1),\,\,\tilde{\Omega}_{+,i}=\tilde{\Omega}_{i}-\mi [R^{w_iw_i}],\,(i\in \mathbb{Z}^d_2),$$
$$A_{+,i}=A_{i}-\mi [R^{z_iw_i}],\,\,\,\tilde{A}_{+,i}=\tilde{A}_{i}-\mi [R^{w_iz_i}],\,(i\in \mathbb{Z}^d_1\cap \mathbb{Z}^d_2).$$

It follows form \eqref{hatN} that
$
\|\hat{N}\|_{s;D_\rho(r,s)\times\mathcal {O}}
\leq \|R\|_{s;D_\rho(r,s)\times\mathcal {O}},
$
then
$|[R^{z^\sigma_iz^\sigma_i}]|_{\mathcal {O}}\leq \|R\|_{s;D_\rho(r,s)\times\mathcal {O}}\leq \varepsilon,$
thus
$$|\Omega_{+,i}-\Omega_{i}|_{\mathcal {O}}=|[R^{z_iz_i}]|_{\mathcal {O}}\leq \varepsilon,\,\hbox{for}\ i\in \mathbb{Z}^d_1.$$
Similarly,
$$|\tilde{\Omega}_{+,i}-\tilde{\Omega}_{i}|_{\mathcal {O}},|A_{+,i}-A_{i}|_{\mathcal {O}},|\tilde{A}_{+,i}-\tilde{A}_{i}|_{\mathcal {O}},|\omega_{+,b}-\omega_{b}|_{\mathcal {O}},|\tilde{\omega}_{+,b}-\tilde{\omega}_{b}|_{\mathcal {O}}\leq \varepsilon.$$

\subsection{The new perturbation}
The new perturbation
\begin{equation*}
  P_+=\int^1_0(\phi^t_F)^*[R(t),F]dt+(\phi^1_F)^*(P-R),
\end{equation*}
with $R(t)=(1-t)[R]+tR.$

Let $\eta=\varepsilon^{\frac{1}{3}}.$  We now give the estimate of $\|P_+\|_{\eta s; D(r-2\delta,\eta s/4)\times\mathcal{O}}.$
\begin{equation*}
  \begin{split}
      \|(\phi^t_F)^*[R(t),F]\|_{\eta s; D(r-2\delta, s/4)\times\mathcal{O}}
      \leq&c \delta^{-1}\eta^{-1}\|R\|_{s; D(r,s)\times\mathcal{O}}\|F\|_{s; D(r,s)\times\mathcal{O}}\\
      \leq&c \delta^{-1}\eta^{-1}\gamma^{-5}K^{5\tau+19}\varepsilon^2.
   \end{split}
\end{equation*}

Consider the estimate for
$\| (\phi^1_{F})^*(P-R)\|_{\eta s; D(r-2\delta,\eta s/4)\times \mathcal{O}}.$
Rewrite $P-R$ as
$$P-R=P_{(1)}+P_{(2)},$$
where
\begin{equation*}
  \begin{split}
      P_{(1)}=&\sum\limits_{\textsf{v}\in \{\theta_b, \varphi_b\}}(1-\mathcal{T}_K)P^{(\textsf{v})}_{000, 000}(\theta,\varphi)\frac{\partial}{\partial \textsf{v}}+\\
     &\sum\limits_{\textsf{v}\in \{I_b, J_b,z^\varrho_j, w^\varrho_j\}}\sum\limits_{ |l|+|\tilde{l}|+\atop |\alpha|+|\beta|+|\tilde{\alpha}|+|\tilde{\beta}|\leq 1}(1-\mathcal{T}_K)P^{(\textsf{v})}_{l \alpha \beta, \tilde{l} \tilde{\alpha} \tilde{\beta}}(\theta,\varphi)I^{l}J^{\tilde{l}}z^{\alpha}\bar{z}^{\beta}w^{\tilde{\alpha}}\bar{w}^{\tilde{\beta}}\frac{\partial}{\partial \textsf{v}}.
   \end{split}
\end{equation*}
\begin{equation*}
  \begin{split}
      P_{(2)}=&\sum\limits_{\textsf{v}\in \{\theta_b, \varphi_b\}}\sum\limits_{ |l|+|\tilde{l}|+\atop |\alpha|+|\beta|+|\tilde{\alpha}|+|\tilde{\beta}|\geq 1}P^{(\textsf{v})}_{l \alpha \beta, \tilde{l} \tilde{\alpha} \tilde{\beta}}(\theta,\varphi)I^{l}J^{\tilde{l}}z^{\alpha}\bar{z}^{\beta}w^{\tilde{\alpha}}\bar{w}^{\tilde{\beta}}\frac{\partial}{\partial \textsf{v}}+\\
     &\sum\limits_{\textsf{v}\in \{I_b, J_b,z^\varrho_j, w^\varrho_j\}}\sum\limits_{ |l|+|\tilde{l}|+\atop |\alpha|+|\beta|+|\tilde{\alpha}|+|\tilde{\beta}|\geq 2}P^{(\textsf{v})}_{l \alpha \beta, \tilde{l} \tilde{\alpha} \tilde{\beta}}(\theta,\varphi)I^{l}J^{\tilde{l}}z^{\alpha}\bar{z}^{\beta}w^{\tilde{\alpha}}\bar{w}^{\tilde{\beta}}\frac{\partial}{\partial \textsf{v}}.
   \end{split}
\end{equation*}
Then
\begin{equation*}
  \begin{split}
      \|P_{(1)}\|_{\eta s; D(r-\delta,\eta s/2)\times\mathcal{O}}
      \leq&\eta^{-1}\me^{-K\delta}  \|P\|_{ s; D(r,s)\times\mathcal{O}}
     \leq\eta^{-1}\me^{-K\delta}  \varepsilon,\\
   \end{split}
\end{equation*}
\begin{equation*}
  \begin{split}
      \|P_{(2)}\|_{\eta s; D(r-\delta,\eta s/2)\times\mathcal{O}}\leq& c \eta \|P\|_{ s; D(r,s)\times\mathcal{O}} \leq c \eta   \varepsilon.\\
   \end{split}
\end{equation*}
This implies that
\begin{equation*}
  \begin{split}
      \|P-R\|_{\eta s; D(r-\delta,\eta s/2)\times\mathcal{O}}
      \leq&\me^{-K\delta}  \|P\|_{\eta s; D(r,\eta s/2)\times\mathcal{O}}+c \eta \|P\|_{ s; D(r,s)\times\mathcal{O}}\\
      \leq&\eta^{-1}\me^{-K\delta}  \varepsilon+c \eta   \varepsilon.\\
   \end{split}
\end{equation*}
Therefore,
\begin{equation*}
\begin{split}
&\|P_+\|_{\eta s; D(r-2\delta,\eta s/4)\times\mathcal{O}}\\
\leq&\|(\phi^t_F)^*[R(t),F]\|_{\eta s; D(r-2\delta,\eta s/4)\times\mathcal{O}}+\|(\phi^1_{F})^*(P-R)\|_{\eta s; D(r-2\delta,\eta s/4)\times \mathcal{O}}
\\
\leq&c \delta^{-1}\eta^{-1}\gamma^{-5}K^{5\tau+19}\varepsilon^2+\eta^{-1}\me^{-K\delta}  \varepsilon+c \eta   \varepsilon\\
=&c \delta^{-1}\gamma^{-5}K^{5\tau+19}\varepsilon^{\frac{5}{3}}+\me^{-K\delta}  \varepsilon^{\frac{2}{3}}+c   \varepsilon^{\frac{4}{3}}\\
\leq&\varepsilon_+.
 \end{split}
\end{equation*}

The following lemma  ensures  that the new perturbation $P_+$ satisfies reversibility and momentum conservation condition.
\begin{lemma}
\begin{enumerate}
  \item $P_+$ is $S-$reversible;
  \item $[P_+,\mathbb{M}_l]=0,\,  l=1,\cdots,d;$
\end{enumerate}
\end{lemma}
\begin{proof}
\begin{enumerate}
\item We conclude from (2) in Lemma \ref{solF} that $\Phi\circ S=S\circ \Phi,$ which implies
$X_+\circ S=-DS\cdot X_+.$ It is obvious that $N_+\circ S=-DS\cdot N_+$,  thus $P_+\circ S=-DS\cdot P_+.$
  \item We know that  $[N,\mathbb{M}_l]=0$ and  $[P,\mathbb{M}_l]=0.$ From Lemma \ref{solF}, we get $[F,\mathbb{M}_l]=0$.
This together with
\begin{equation}\label{PTaye}
\begin{split}
P_+=&P-R+[P,F]+\frac{1}{2!}[[N,F],F]+\frac{1}{2!}[[P,F],F]\\
   &+\cdots+\frac{1}{i!}[\cdots[N,\underbrace{F]\cdots,F}_i]+\frac{1}{i!}[\cdots[P,\underbrace{F]\cdots,F}_i]+\cdots
 \end{split}
\end{equation}
implies that  $[P_+,\mathbb{M}_l]=0$.
\end{enumerate}
\end{proof}

\begin{lemma}\label{LieBTL}
Suppose $P$ satisfies (A6),  $F$ satisfies (A6) with $\varepsilon^{\frac{2}{3}}$ in place of $\varepsilon.$ Moreover,
for $\sigma=\pm,$ $|j|>K,$
$$\frac{\partial F^{(I_b)}}{\partial z^\sigma_{j} }=0, \,\, \frac{\partial F^{(I_b)}}{\partial w^\sigma_{j}}=0,\,\, \frac{\partial F^{(J_b)}}{\partial z^\sigma_{j} }=0, \,\, \frac{\partial F^{(J_b)}}{\partial w^\sigma_{j}}=0
$$
and for $|i\mp j|>K,$
$$\frac{\partial F^{(z^\sigma_i)}}{\partial z^{\pm\sigma}_j}=0,\,\,
\frac{\partial F^{(w^\sigma_i)}}{\partial w^{\pm\sigma}_j}=0,\,\,
\frac{\partial F^{(z^\sigma_i)}}{\partial w^{\pm\sigma}_j}=0,\,\,
\frac{\partial F^{(w^\sigma_i)}}{\partial z^{\pm\sigma}_j}=0,$$
then $[P, F]$ also satisfies (A6) with $\varepsilon_+$ in place of $\varepsilon.$
\end{lemma}
\begin{proof}
By the definition of Lie bracket, the $z_i$-component of $[P, F]$  is
   \begin{equation*}
    [P, F]^{(z_i)}=\sum_{\textsf{u} \in \mathscr{V}}(\frac{\partial P^{(z_i)}}{\partial \textsf{u}}F^{(\textsf{u})}-\frac{\partial F^{(z_i)}}{\partial \textsf{u}}P^{(\textsf{u})}),
  \end{equation*}
 where
 $\mathscr{V}=\{\theta_a, \varphi_b,\,z_i, w_j,\bar{z}_i,\bar{w}_j: a=1,\cdots,n;b=1,\cdots,m;i\in\mathbb{Z}^d_1;j\in\mathbb{Z}^d_2\}.$

To verify  $[P, F]$ satisfies (A6),  we only consider  $\frac{\partial}{\partial z_{j}}[P, F]^{(z_{i})}$
and  the derivatives with respect to the other components are similarly analyzed.

In the following,
 it suffices  to  consider
$\sum\limits_{h}\frac{\partial^2 P^{(z_i)}}{\partial z_h \partial z_j}F^{(z_h)}$
and
  $\sum\limits_{h}\frac{\partial P^{(z_i)}}{\partial z_h }\frac{\partial F^{(z_h)}}{\partial z_j }$
in $\frac{\partial}{\partial z_{j}}[P, F]^{(z_{i})}$ since  the other terms can be similarly studied.

Let $ p^{zz}_{ij}=\lim\limits_{t\rightarrow\infty}\frac{\partial P^{(z_{i+tc})}}{\partial z_{j+tc}},\,\,f^{zz}_{ij}=\lim\limits_{t\rightarrow\infty}\frac{\partial F^{(z_{i+tc})}}{\partial z_{j+tc}}.$
Then
\begin{equation}
\begin{split}
&\|\sum\limits_{h}(\frac{\partial^2 F^{(z_{i+tc})}}{\partial z_{h} \partial z_{j+tc}}F^{(z_{h})}-\lim\limits_{t\rightarrow\infty}\frac{\partial^2 F^{(z_{i+tc})}}{\partial z_{h} \partial z_{j+tc}}F^{(z_{h})})\|_{s;D_\rho{(r-\delta,s/2)}}\\
\leq&\|F\|_{s;D_\rho{(r,s)}}\|\frac{\partial P^{(z_{i+tc})}}{ \partial z_{j+tc}}- p^{zz}_{ij}\|_{s;D_\rho{(r,s)}}\\
\leq&c\frac{\varepsilon^{\frac{5}{3}}}{|t|}\me^{-\rho |i-j|}\leq\frac{\varepsilon_+}{30|t|}\me^{-\rho_+ |i-j|}.\\
 \end{split}
\end{equation}

\begin{equation}
\begin{split}
&\|\sum\limits_{h}(\frac{\partial P^{(z_{i+tc})}}{\partial z_{h+tc} }\frac{\partial F^{(z_{h+tc})}}{\partial z_{j+tc} }-p^{zz}_{ih}f^{zz}_{hj})\|_{s;D_\rho{(r-\delta,s/2)}}\\
\leq&\sum\limits_{h}\|f^{zz}_{hj}\|_{s;D_\rho{(r-\delta,s/2)}}\|\frac{\partial P^{(z_{i+tc})}}{\partial z_{h+tc} }-p^{zz}_{ih}\|_{s;D_\rho{(r-\delta,s/2)}}\\
&+\sum\limits_{h}\|p^{zz}_{ih}\|_{s;D_\rho{(r-\delta,s/2)}}\|\frac{\partial F^{(z_{h+tc})}}{\partial z_{j+tc} }-f^{zz}_{hj}\|_{s;D_\rho{(r-\delta,s/2)}}\\
&+\sum\limits_{h}\|\frac{\partial  P^{(z_{i+tc})}}{\partial z_{h+tc} }-p^{zz}_{ih}\|_{s;D_\rho{(r-\delta,s/2)}}\|\frac{\partial F^{(z_{h+tc})}}{\partial z_{j+tc} }-f^{zz}_{hj}\|_{s;D_\rho{(r-\delta,s/2)}}\\
\leq& cK^d\frac{\varepsilon^{\frac{5}{3}}}{|t|}\me^{-\rho |i-j|}+cK^d\frac{\varepsilon^{\frac{5}{3}}}{t^2}\me^{-\rho |i-j|}\\
\leq&\frac{\varepsilon_+}{30|t|}\me^{-\rho_+ |i-j|}.\\
 \end{split}
\end{equation}
Note that $h$ is bounded by $cK^d$ in the above inequality since $|i-h|\leq K$ and $|j-h|\leq K.$

\end{proof}

By \eqref{PTaye} and Lemma \ref{LieBTL} above, we obtain:
\begin{lemma}\label{PTL}
$P_+$ satisfies (A6) with $K_+, \varepsilon_+, \rho_+$ in place of $K, \varepsilon, \rho$.
\end{lemma}

\subsection{Iteration and Convergence}

For given  $r>0,s>0, L>0$  and $0<\gamma,\varepsilon<1.$ $c$ is a positive constant  depending only on $n,m,\tau.$
For $\nu \geq 0,$ we define the iterative  sequences as follows:
$$\delta_\nu=\frac{r}{2^{\nu+3}},\,\,r_{\nu+1}=r_{\nu}-2\delta_\nu,\,\, r_0=r,$$
$$\varepsilon_{\nu+1}=c\gamma^{-5}\delta^{-1}_\nu K^{5\tau+19}_\nu\varepsilon^{\frac{5}{3}}_\nu+\varepsilon^{\frac{7}{6}}_\nu,\,\, \varepsilon_0=\varepsilon,$$
$$\me^{-K_\nu\delta_\nu}=\varepsilon^{\frac{1}{2}}_\nu,$$
$$\eta_\nu=\varepsilon^{\frac{1}{3}}_\nu,\,\,s_{\nu+1}=\frac{1}{4}\eta_\nu s_{\nu},\,\,s_0=s,$$
$$L_{\nu+1}=L_{\nu}+\varepsilon_\nu,\,\, L_{0}=L,$$
$$\rho_{\nu}=\rho(1-\sum\limits^{\nu+1}_{i=2}2^{-i}).$$

\subsubsection{Iteration lemma}
According to the preceding analysis, we obtain the following iterative lemma.
\begin{lemma}[Iterative Lemma]\label{Itelem}
Let $\varepsilon$ be small enough and $\nu\geq0.$ Suppose that
\begin{enumerate}
  \item The normal form $$N_\nu+\mathcal{A}_\nu= \omega_\nu(\zeta)\frac{\partial}{\partial \theta}+\tilde{\omega}_\nu(\zeta)\frac{\partial}{\partial \varphi}+\sum\limits_{\sigma=\pm}\sigma\mi(\Omega_\nu(\zeta) z^\sigma\frac{\partial}{\partial z^\sigma}
  +\widetilde{\Omega}_\nu(\zeta)w^\sigma\frac{\partial}{\partial w^\sigma})+\mathcal{A}_\nu,$$  $(\zeta\in \mathcal{O}_\nu)$ satisfies
\eqref{MelK} with  $\omega_\nu,  \tilde{\omega}_\nu, \Omega_{\nu},  \tilde{\Omega}_{\nu}, A_{\nu},  \tilde{A}_{\nu}$ and $K_\nu;$

  \item $\omega_\nu,  \tilde{\omega}_\nu, \Omega_{\nu,j},  \tilde{\Omega}_{\nu,j}$ are $C^4_W$ smooth in $\zeta$ and satisfy
$$|\omega_{\nu}-\omega_{\nu-1}|_{\mathcal {O}},|\tilde{\omega}_{\nu}-\tilde{\omega}_{\nu-1}|_{\mathcal {O}}, |\Omega_{\nu,j}-\Omega_{\nu-1,j}|_{\mathcal {O}},|\tilde{\Omega}_{\nu,j}-\tilde{\Omega}_{\nu-1,j}|_{\mathcal {O}}\leq \varepsilon_{\nu-1};$$
  \item $N_\nu+\mathcal{A}_\nu+P_\nu$ satisfies $(A5)$ and $(A6)$ with $K_\nu, \varepsilon_\nu, \rho_\nu$ and
  $$\|P_\nu\|_{ s_\nu; D_{\rho_\nu}(r_\nu,s_\nu)\times\mathcal{O}_\nu}\leq\varepsilon_\nu.$$
\end{enumerate}
 Then there exists a real analytic, $S-$invariant transformation
    $$\Phi_\nu: D_{\rho_{\nu}}(r_{\nu+1},s_{\nu+1})\times\mathcal{O}_\nu\rightarrow D_{\rho_{\nu}}(r_{\nu},s_{\nu})$$
satisfying
\begin{equation}\label{Phi}
        \|\Phi_\nu-id\|_{s_{\nu+1}; D_{\rho_{\nu}}(r_{\nu+1},s_{\nu+1})\times\mathcal{O}_\nu}\leq c\varepsilon^{\frac{1}{2}}_\nu,
\end{equation}
\begin{equation}\label{DPhi}
        \|D\Phi_\nu-Id\|_{s_{\nu+1}; D_{\rho_{\nu}}(r_{\nu+1},s_{\nu+1})\times\mathcal{O}_\nu}\leq c\varepsilon^{\frac{1}{2}}_\nu,
\end{equation}
    and a closed subset
 \begin{equation}\label{newnonresonset}
    \mathcal{O}_{\nu+1}=\mathcal{O}_\nu \setminus \bigcup_{K_\nu<|k|+|\tilde{k}|\leq K_{\nu+1}}\mathcal{R}^{\nu+1}_{k\tilde{k}}(\gamma),
\end{equation}
where $\mathcal{R}^{\nu+1}_{k\tilde{k}}(\gamma)$ is defined in \eqref{Res},
such that  $X_{\nu+1}=(\Phi_{\nu})^*X_\nu=N_{\nu+1}+\mathcal{A}_{\nu+1}+P_{\nu+1}$
satisfies the same assumptions as $X_\nu$
 with `$\nu+1$' in place of `$\nu$'.
\end{lemma}

\subsubsection{Convergence}

We now finish the proof of Theorem \ref{KAM}.
Let
\begin{equation*}
\begin{split}
  X_{0}=&N_{0}+\mathcal{A}_{0}+P_{0}\\
  =&N+\mathcal{A}+P\\
\end{split}
\end{equation*}
be initial $S-$reversible vector field and satisfies the assumptions of Theorem \ref{KAM}.
Recall that $\varepsilon_0=\varepsilon,$ $r_0=r,$ $s_0=s,$ $\rho_0=\rho,$ $L_0=L.$
Suppose $\mathcal{O}$ is a compact set of positive Lebesgue measure and
all the conditions  in the  iterative lemma with $\nu=0$ hold. Then we inductively obtain the following sequences
$$ \mathcal{O}_{\nu+1}\subset\mathcal{O}_\nu,$$
$$\Psi^\nu=\Phi_0\circ\Phi_1\circ\cdots\circ\Phi_\nu: D_{\rho_{\nu}}(r_{\nu+1},s_{\nu+1})\times\mathcal{O}_\nu\rightarrow D_{\rho_{0}}(r_{0},s_{0}),$$
$$ X_{\nu+1}=(\Psi^{\nu})^*X=N_{\nu+1}+\mathcal{A}_{\nu+1}+P_{\nu+1}.$$

Let $\widetilde{\mathcal{O}}=\cap^\infty_{\nu=0} \mathcal{O}_{\nu}.$
Using \eqref{Phi}, \eqref{DPhi} and  following from  \cite{Poschel89}, we obtain that
  $N_\nu+A_\nu, \Psi^{\nu}, D\Psi^{\nu}$ converge uniformly on $D_{\frac{\rho}{2}}(\frac{r}{2},0)\times \widetilde{\mathcal{O}}$
  with
  \begin{equation*}
\begin{split}
 N_\infty+\mathcal{A}_\infty=&\omega_\infty\frac{\partial}{\partial \theta}+\tilde{\omega}_\infty\frac{\partial}{\partial \varphi}+\\
  &\sum\limits_{\sigma=\pm}\sigma\mi(\Omega_\infty z^\sigma\frac{\partial}{\partial z^\sigma}
  +\widetilde{\Omega}_\infty w^\sigma\frac{\partial}{\partial w^\sigma})+\mathcal{A}_\infty.
\end{split}
\end{equation*}
By the choice of $\varepsilon_\nu$ and $K_\nu,$ we have  $\varepsilon_{\nu+1}=O(\varepsilon^{\frac{7}{6}}_\nu),$ thus $\varepsilon_\nu\rightarrow0, \nu\rightarrow \infty.$
And we also have $\sum^\infty_{\nu=0}\varepsilon_\nu\leq2\varepsilon.$
Consider the flow $\phi^t_X$ of $X.$ It follows from $X_{\nu+1}=(\Psi^{\nu})^*X$ that
\begin{equation}\label{equa}
    \phi^t_X\circ \Psi^{\nu}=\Psi^{\nu}\circ\phi^t_{X_{\nu+1}}.
\end{equation}
Thanks to the uniform converge  of $X_\nu, \Psi^{\nu}\ \hbox{and}\ D\Psi^{\nu},$ we can take the limits on both
 sides of \eqref{equa}. Therefore,  on  $D_{\frac{\rho}{2}}(\frac{r}{2},0)\times \widetilde{\mathcal{O}}$,
we have
\begin{equation*}
    \phi^t_X\circ \Psi^{\infty}=\Psi^{\infty}\circ\phi^t_{X_{\infty}}
\end{equation*}
and
 $$\Psi^{\infty}: D_{\frac{\rho}{2}}(\frac{r}{2},0)\times \widetilde{\mathcal{O}}\rightarrow D_{\rho}(r,s)\times \mathcal{O}.$$
It follows that  for each $\zeta\in \widetilde{\mathcal{O}},$
$ \Psi^{\infty}(\mathbb{T}^{n+m}\times\{\zeta\})$ is and embedded torus which is invariant for the original
perturbed reversible system at $\zeta\in \widetilde{\mathcal{O}}.$

\subsection{Measure Estimate}
Let $\mathcal{O}_{-1}=\mathcal{O},$ $K_{-1}=0.$
At the $\nu$th step of KAM iteration, the following resonant set $\mathcal{R}^{\nu}\subset \mathcal{O}_{\nu-1}$ need to be excluded.
\begin{equation}\label{Res}
    \mathcal{R}^{\nu}=\bigcup\limits_{K_{\nu-1}<|k|+|\tilde{k}|\leq K_{\nu}}\mathcal{R}^{\nu}_{k\tilde{k}},
\end{equation}
 with
 \begin{equation*}
\begin{split}
\mathcal{R}^{\nu}_{k\tilde{k}}=& \mathcal{R}^{0\nu}_{k\tilde{k}}\cup (\bigcup\limits_{i}\mathcal{R}^{1\nu}_{k\tilde{k},i})\cup (\bigcup\limits_{i}\mathcal{R}^{2\nu}_{k\tilde{k},i})
\cup (\bigcup\limits_{ij}\mathcal{R}^{11,\pm\nu}_{k\tilde{k},ij})\cup (\bigcup\limits_{ij}\mathcal{R}^{12,\pm\nu}_{k\tilde{k},ij})\cup\\
& (\bigcup\limits_{ij}\mathcal{R}^{22,\pm\nu}_{k\tilde{k},ij})\cup (\bigcup\limits_{i}\mathcal{R}^{3\nu}_{k\tilde{k},i})\cup (\bigcup\limits_{ij}\mathcal{R}^{13,\pm\nu}_{k\tilde{k},ij})\cup (\bigcup\limits_{ij}\mathcal{R}^{23,\pm\nu}_{k\tilde{k},ij})\cup (\bigcup\limits_{ij}\mathcal{R}^{34,\pm\nu}_{k\tilde{k},ij}),
\end{split}
\end{equation*}
where
\begin{equation}\label{R0}
    \mathcal{R}^{0\nu}_{k\tilde{k}}=\{\zeta\in \mathcal{O}_{\nu-1}:|\langle k,\omega_\nu(\zeta)\rangle+\langle \tilde{k},\tilde{\omega}_\nu(\zeta)\rangle|<\frac{\gamma}{K_\nu^\tau}\},
\end{equation}
\begin{equation}\label{R1}
    \mathcal{R}^{1\nu}_{k\tilde{k},i}=\{\zeta: |\langle k,\omega_\nu\rangle+\langle \tilde{k},\tilde{\omega}_\nu\rangle + \Omega_{\nu,i}|<\frac{\gamma}{K_\nu^\tau}\},
\end{equation}
$i\in \mathbb{Z}^d_1.$
\begin{equation}\label{R2}
    \mathcal{R}^{2\nu}_{k\tilde{k},i}=\{\zeta: |\langle k,\omega_\nu\rangle+\langle \tilde{k},\tilde{\omega}_\nu\rangle + \tilde{\Omega}_{\nu,i}|<\frac{\gamma}{K_\nu^\tau}\},
\end{equation}
$i\in \mathbb{Z}^d_2.$
\begin{equation}\label{R11}
    \mathcal{R}^{11,\pm\nu}_{k\tilde{k},ij}=\{\zeta: |\langle k,\omega_\nu\rangle+\langle \tilde{k},\tilde{\omega}_\nu\rangle+ \Omega_{\nu,i}\pm \Omega_{\nu,j}|<\frac{\gamma}{K_\nu^\tau}\},
\end{equation}
$i,j\in \mathbb{Z}^d_1\setminus\mathbb{Z}^d_2.$
\begin{equation}\label{R12}
    \mathcal{R}^{12,\pm\nu}_{k\tilde{k},ij}=\{\zeta: |\langle k,\omega_\nu\rangle+\langle \tilde{k},\tilde{\omega}_\nu\rangle+ \Omega_{\nu,i}\pm\tilde{\Omega}_{\nu,j}|<\frac{\gamma}{K_\nu^\tau}\},
\end{equation}
$i\in \mathbb{Z}^d_1\setminus\mathbb{Z}^d_2 ,j\in \mathbb{Z}^d_2\setminus\mathbb{Z}^d_1.$
\begin{equation}\label{R22}
    \mathcal{R}^{22,\pm\nu}_{k\tilde{k},ij}=\{\zeta: |\langle k,\omega_\nu\rangle+\langle \tilde{k},\tilde{\omega}_\nu\rangle+ \tilde{\Omega}_{\nu,i}\pm\tilde{\Omega}_{\nu,j}|<\frac{\gamma}{K_\nu^\tau}\},
\end{equation}
$i,j\in \mathbb{Z}^d_2\setminus\mathbb{Z}^d_1.$
\begin{equation}\label{R3}
    \mathcal{R}^{3\nu}_{k\tilde{k},i}=\{\zeta: |\det((\langle k,\omega_\nu\rangle+\langle \tilde{k},\tilde{\omega}_\nu\rangle)I_2+ M_{\nu,i})|<\frac{\gamma}{K_\nu^\tau}\},
\end{equation}
$i\in \mathbb{Z}^d_1\cap\mathbb{Z}^d_2.$
\begin{equation}\label{R13}
    \mathcal{R}^{13,\pm\nu}_{k\tilde{k},ij}=\{\zeta: |\det((\langle k,\omega_\nu\rangle+\langle \tilde{k},\tilde{\omega}_\nu\rangle+ \Omega_{\nu,i})I_2\pm M_{\nu,j})|<\frac{\gamma}{K_\nu^\tau}\},
\end{equation}
$(i,j)\ \hbox{or}\ (i,j)\in (\mathbb{Z}^d_1\cap\mathbb{Z}^d_2)\times(\mathbb{Z}^d_1\setminus\mathbb{Z}^d_2).$
\begin{equation}\label{R23}
    \mathcal{R}^{23,\pm\nu}_{k\tilde{k},ij}=\{\zeta: |\det((\langle k,\omega_\nu\rangle+\langle \tilde{k},\tilde{\omega}_\nu\rangle+ \tilde{\Omega}_{\nu,i})I_2\pm M_{\nu,j})|<\frac{\gamma}{K_\nu^\tau}\},
\end{equation}
$(i,j)\ \hbox{or}\ (i,j)\in (\mathbb{Z}^d_1\cap\mathbb{Z}^d_2)\times(\mathbb{Z}^d_2\setminus\mathbb{Z}^d_1).$
\begin{equation}\label{R34}
    \mathcal{R}^{34,\pm\nu}_{k\tilde{k},ij}=\{\zeta:|\det((\langle k,\omega_\nu\rangle+\langle \tilde{k},\tilde{\omega}_\nu\rangle)I_{4}+ M_{\nu,i}\otimes I_2\pm  I_2\otimes M^T_{\nu,j})|<\frac{\gamma}{K_\nu^\tau}\},\\
    i,j\in \mathbb{Z}^d_1\cap\mathbb{Z}^d_2.
\end{equation}

Note that $\mathcal{R}^{11,-\nu}_{k\tilde{k},ij},$  $\mathcal{R}^{22,-\nu}_{k\tilde{k},ij}$ and  $\mathcal{R}^{34,-\nu}_{k\tilde{k},ij}$ are the most complicated three case, and the former two have been studied in \cite{Geng13}, thus it suffices to  consider the last case .

Denote $D_\nu=(\langle k,\omega_\nu\rangle+\langle \tilde{k},\tilde{\omega}_\nu\rangle)I_{4}+ M_{\nu,i}\otimes I_2- I_2\otimes M^T_{\nu,j}.$

\begin{lemma}\label{mea1}
For any given $i,j\in \mathbb{Z}^d_1\cap \mathbb{Z}^d_2$ with $|i-j|\leq K_\nu,$
 either $|\det (D_\nu)|\geq 1$ or there are $i_0, j_0, c_1, c_2, \cdots , c_{d-1} \in  \mathbb{Z}^d$ with
$|i_0|, |j_0|, |c_1|, |c_2|,\cdots , |c_{d-1}|\leq3K^2_\nu$ and $t_1, t_2,\cdots, t_{d-1} \in  \mathbb{Z}$
such that $i=i_0+ t_1c_1+ t_2c_2+\cdots+t_{d-1} c_{d-1}, j=j_0+ t_1c_1+ t_2c_2+\cdots+t_{d-1} c_{d-1}.$
\end{lemma}
\begin{proof}
See \cite{Geng13}.
\end{proof}

In the following, for  the convenience of notation, let $\textbf{t}:=(t_1, t_2,\cdots, t_{d-1}),$ $\textbf{c}:=(c_1, c_2,\cdots, c_{d-1})$
 and $\textbf{t}\cdot\textbf{c}:=t_1c_1+ t_2c_2+\cdots+t_{d-1} c_{d-1}.$

By Lemma \ref{mea1}, we have
\begin{lemma}\label{mea2}
$$\bigcup\limits_{i,j\in \mathbb{Z}^d_1\cap \mathbb{Z}^d_2 }\mathcal{R}^{34,-\nu}_{k\tilde{k},ij}\subset \bigcup\limits_{i_0, j_0, c_1, c_2, \cdots , c_{d-1} \in  \mathbb{Z}^d;\atop t_1, t_2,\cdots, t_{d-1} \in  \mathbb{Z}}\mathcal{R}^{34,-\nu}_{k\tilde{k},i_0+\textbf{t}\cdot\textbf{c},j_0+\textbf{t}\cdot\textbf{c}},$$
 where
$|i_0|, |j_0|, |c_1|, |c_2|,\cdots , |c_{d-1}|\leq3K^2_\nu$.
\end{lemma}

\begin{lemma}\label{mea3}
Let $\tau>\frac{4(d-1)(d+1)!}{(d-1)!(d+1)-1}.$ Then for fixed $k, \tilde{k}, i_0, j_0, c_1, c_2,\cdots , c_{d-1}$,
$$ \meas(\bigcup\limits_{ t_1, t_2,\cdots, t_{d-1} \in  \mathbb{Z}}\mathcal{R}^{34,-\nu}_{k\tilde{k},i_0+\textbf{t}\cdot\textbf{c},j_0+\textbf{t}\cdot\textbf{c}})\leq c\frac{\gamma^{\frac 14}}{K_\nu^{\frac{\tau}{(d+1)!}}}.$$
\end{lemma}
\begin{proof}
Without loss of generality, we assume $|t_1|\leq|t_2|\leq \cdots\leq|t_{d-1}|.$
Let $\Omega_{\nu,j}=|j|^2+\Omega^0_{\nu,j},\,\,\tilde{\Omega}_{\nu,j}=|j|^2+\tilde{\Omega}^0_{\nu,j}$
and $D_\nu(\textbf{t})=(\langle k,\omega_\nu\rangle+\langle \tilde{k},\tilde{\omega}_\nu\rangle)I_{4}+ M_{\nu,i_0+\textbf{t}\cdot\textbf{c}}\otimes I_2- I_2\otimes M^T_{\nu,j_0+\textbf{t}\cdot\textbf{c}}.$

Using T\"{o}plitz-Lipschitz property of $\mathcal{A}_\nu+P_\nu,$ we have  for $l=i_0, j_0, 1\leq j\leq d-1,$
$$|\Omega^0_{\nu,l+\textbf{t}\cdot\textbf{c}}-\lim_{t_j\rightarrow\infty}\Omega^0_{\nu,l+\textbf{t}\cdot\textbf{c}}|<\frac{\varepsilon}{|t_j|},$$
$$|\tilde{\Omega}^0_{\nu,l+\textbf{t}\cdot\textbf{c}}-\lim_{t_j\rightarrow\infty}\tilde{\Omega}^0_{\nu,l+\textbf{t}\cdot\textbf{c}}|<\frac{\varepsilon}{|t_j|},$$
$$|A_{\nu,l+\textbf{t}\cdot\textbf{c}}-\lim_{t_j\rightarrow\infty}A_{\nu,l+\textbf{t}\cdot\textbf{c}}|<\frac{\varepsilon}{|t_j|},$$
$$|\tilde{A}_{\nu,l+\textbf{t}\cdot\textbf{c}}-\lim_{t_j\rightarrow\infty}\tilde{A}_{\nu,l+\textbf{t}\cdot\textbf{c}}|<\frac{\varepsilon}{|t_j|},$$
then we have
$$|\det(D_\nu(\textbf{t}))-\lim_{t_j\rightarrow\infty}\det(D_\nu(\textbf{t}))|<\frac{\varepsilon K^4_\nu}{|t_j|}.$$

Consider the resonant set
\begin{equation*}
    \mathcal{R}^{34,-\nu}_{k\tilde{k},i_0j_0\textbf{c}\infty^{d-1}}=\{\zeta\in \mathcal{O}_{\nu-1}:|\lim_{t_1\rightarrow\infty}(\lim_{t_2,\cdots,t_{d-1}\rightarrow\infty}\det(D_\nu(\textbf{t})))|<\frac{\gamma}{K_\nu^{\frac{\tau}{d!}}}\}.
\end{equation*}
For fixed $k, \tilde{k},i_0, j_0, \textbf{c},$ its Lebesgue measure
$$\meas(\mathcal{R}^{34,-\nu}_{k\tilde{k},i_0j_0\textbf{c}\infty^{d-1}})\leq \frac{\gamma^{\frac 14}}{K_\nu^{\frac{\tau}{d!}}},$$
and for $\zeta\in \mathcal{O}_{\nu-1}\setminus  \mathcal{R}^{34,-\nu}_{k\tilde{k},i_0j_0\textbf{c}\infty^{d-1}},$
$$|\lim_{t_1\rightarrow\infty}(\lim_{t_2,\cdots,t_{d-1}\rightarrow\infty}\det(D_\nu(\textbf{t})))|\geq\frac{\gamma}{K_\nu^{\frac{\tau}{d!}}}.$$

Below we consider the following $d$ cases.

{\bf Case (1): $|t_1|>K_\nu^{\frac{\tau}{d!}+4}$.}
For $\zeta\in \mathcal{O}_{\nu-1}\setminus  \mathcal{R}^{34,-\nu}_{k\tilde{k},i_0j_0\textbf{c}\infty^{d-1}},$ we have
 \begin{equation*}
\begin{split}
 |\det(D_\nu(\textbf{t}))|
  \geq& |\lim_{t_1\rightarrow\infty}(\lim_{t_2,\cdots,t_{d-1}\rightarrow\infty}\det(D_\nu(\textbf{t})))|-\sum^{d-1}_{j=1}\frac{\varepsilon K^4_\nu}{|t_j|}\\
  \geq& \frac{\gamma}{K_\nu^{\frac{\tau}{d!}}}-(d-1)\frac{\varepsilon }{K_\nu^{\frac{\tau}{d!}}}\\
  \geq& \frac{\gamma}{2K_\nu^{\frac{\tau}{d!}}}\geq\frac{\gamma}{K_\nu^{\tau}}.
\end{split}
\end{equation*}

In general,
for  $2\leq l\leq d-1,$ consider

{\bf Case (l): $|t_1|\leq K_\nu^{\frac{\tau}{d!}+4}, |t_2|\leq K_\nu^{\frac{2\tau}{d!}+4},\cdots,|t_{l-1}|\leq K_\nu^{\frac{(l-1)!\tau}{d!}+4}, |t_l|> K_\nu^{\frac{l!\tau}{d!}+4}$}.
We define the resonant set
\begin{equation*}
    \mathcal{R}^{34,-\nu}_{k\tilde{k},i_0j_0\textbf{c}t_1t_2\cdots t_{l-1}\infty^{d-l}}=\{\zeta\in \mathcal{O}_{\nu-1}:|\lim_{t_l,\cdots,t_{d-1}\rightarrow\infty}\det(D_\nu(\textbf{t}))|<\frac{\gamma}{K_\nu^{\frac{l!\tau}{d!}}}\}.
\end{equation*}
Then for fixed $k, \tilde{k},i_0, j_0, \textbf{c},t_1,t_2,\cdots,t_{l-1},$ its Lebesgue measure
$$\meas(\mathcal{R}^{34,-\nu}_{k\tilde{k},i_0j_0\textbf{c}t_1\cdots t_{l-1}\infty^{d-l}})\leq \frac{\gamma^{\frac 14}}{K_\nu^{\frac{l!\tau}{d!}}},$$
and
 \begin{equation*}
\begin{split}
&\meas(\bigcup\limits_{|t_1|,\cdots,|t_{l-1}|
\leq K_\nu^{\frac{(l-1)!\tau}{d!}+4}}\mathcal{R}^{34,-\nu}_{k\tilde{k},i_0j_0\textbf{c}t_1\cdots t_{l-1}\infty^{d-l}})\\
\leq &2^{l-1}K_\nu^{\frac{(l-1)!(l-1)\tau}{d!}+4(l-1)}\frac{\gamma^{\frac 14}}{K_\nu^{\frac{l!\tau}{d!}}}
\leq\frac{2^{l-1}\gamma^{\frac 14}}{K_\nu^{\frac{(l-1)!\tau}{d!}-4(l-1)}}.
\end{split}
\end{equation*}
Thus for $|t_1|\leq K_\nu^{\frac{\tau}{d!}+4}, |t_2|\leq K_\nu^{\frac{2\tau}{d!}+4},\cdots,|t_{l-1}|\leq K_\nu^{\frac{(l-1)!\tau}{d!}+4}, |t_l|> K_\nu^{\frac{l!\tau}{d!}+4}$, $\zeta\in \mathcal{O}_{\nu-1}\setminus \mathcal{R}^{34,-\nu}_{k\tilde{k},i_0j_0\textbf{c}t_1\cdots t_{l-1}\infty^{d-l}},$
\begin{equation*}
\begin{split}
 |\det(D_\nu(\textbf{t}))|
  \geq& |\lim_{t_l,\cdots,t_{d-1}\rightarrow\infty}\det(D_\nu(\textbf{t}))|-\sum^{d-1}_{j=l}\frac{\varepsilon K^4_\nu}{|t_j|}\\
  \geq& \frac{\gamma}{K_\nu^{\frac{l!\tau}{d!}}}-(d-l)\frac{\varepsilon }{K_\nu^{\frac{l!\tau}{d!}}}\\
  \geq& \frac{\gamma}{2K_\nu^{\frac{l!\tau}{d!}}}\geq\frac{\gamma}{K_\nu^{\tau}}.
\end{split}
\end{equation*}

At last, for

{\bf Case (d): $|t_1|\leq K_\nu^{\frac{\tau}{d!}+4}, |t_2|\leq K_\nu^{\frac{2\tau}{d!}+4},\cdots,|t_{d-1}|\leq K_\nu^{\frac{(d-1)!\tau}{d!}+4}$.}
We define the resonant set
\begin{equation*}
    \mathcal{R}^{34,-\nu}_{k\tilde{k},i_0j_0\textbf{c}t_1t_2\cdots t_{d-1}}=\{\zeta\in \mathcal{O}_{\nu-1}:|\det(D_\nu(\textbf{t}))|<\frac{\gamma}{K_\nu^{\tau}}\}.
\end{equation*}
For fixed $k, \tilde{k},i_0, j_0, \textbf{c},t_1,t_2,\cdots,t_{d-1},$ its Lebesgue measure
$$\meas(\mathcal{R}^{34,-\nu}_{k\tilde{k},i_0j_0\textbf{c}t_1t_2\cdots t_{d-1}})\leq \frac{\gamma^{\frac 14}}{K_\nu^{\tau}},$$
and
 \begin{equation*}
\begin{split}
&\meas(\bigcup\limits_{|t_1|,\cdots,|t_{d-1}|\leq K_\nu^{\frac{(d-1)!\tau}{d!}+4}}\mathcal{R}^{34,-\nu}_{k\tilde{k},i_0j_0\textbf{c}t_1t_2\cdots t_{d-1}})\\
\leq &2^{d-1}K_\nu^{\frac{(d-1)!(d-1)\tau}{d!}+4(d-1)}\frac{\gamma^{\frac 14}}{K_\nu^{\tau}}
\leq\frac{2^{d-1}\gamma^{\frac 14}}{K_\nu^{\frac{\tau}{d}-4(d-1)}}.
\end{split}
\end{equation*}

Therefore, if $\tau>\frac{4(d-1)(d+1)!}{(d-1)!(d+1)-1},$ we obtain
$$ \meas(\bigcup\limits_{ t_1, t_2,\cdots, t_{d-1} \in  \mathbb{Z}}\mathcal{R}^{34,-\nu}_{k\tilde{k},i_0+\textbf{t}\cdot\textbf{c},j_0+\textbf{t}\cdot\textbf{c}})\leq c\frac{\gamma^{\frac 14}}{K_\nu^{\frac{\tau}{(d+1)!}}}.$$
\end{proof}

According to the above analysis , we obtain the following lemma.
\begin{lemma}\label{mea4}
Let $\tau>d!(2d(d+1)+n+m+1)+ \frac{4(d-1)(d+1)!}{(d-1)!(d+1)-1}.$ Then the total measure of resonant set should be excluded during the KAM iteration is
$$ \meas(\bigcup\limits_{ \nu\geq0}\mathcal{R}^{\nu})=O(\gamma^{\frac 14}).$$
\end{lemma}

\section{Acknowledgments}
The research was  supported by National Natural Science Foundation of China (Grant No. 11271180).

\section{Appendix}

Suppose vector field $X(\theta, I, z, \bar{z})$ is defined on $D_\rho(r,s)=\{y=(\theta, I, z, \bar{z}):|\textmd{Im} \theta|<r, |I|<s,
   \|z\|_{\rho}<s, \|\bar{z}\|_{\rho}<s\}.$
\begin{definition}[Reversible vector field ]\label{def/revvf}
  Suppose $S$ is an involution map: $S^2=id.$ Vector field  $X$ is called reversible with respect to
$S$ (or $S-$reversible),  if
  $$DS\cdot X=- X\circ S,$$ i.e.,
  $$(DS(y))X(y)=-X(S(y)),y\in D_\rho(r,s),$$
where $DS$ is the tangent map of $S.$

\end{definition}

\begin{definition}
  Suppose $S$ is an involution map: $S^2=id.$  Vector field $X$ is called
   invariant with respect to
$S$ (or $S-$invariant),  if
  $$ DS\cdot X=X\circ S.$$
\end{definition}

\begin{definition}
A transformation   $\Phi$ is called invariant with respect to  above involution
$S$ (or $S-$invariant),  if  $\Phi\circ S=S\circ \Phi.$
\end{definition}

\begin{lemma}
\begin{enumerate}
  \item If $X$ and  $Y$ are both  $S-$reversible (or $S-$invariant), then $[X, Y]$  is  $S-$invariant.
  \item If $X$ is  $S-$reversible,   $Y$  is $S-$invariant and the transformation   $\Phi$
is $S-$invariant, then $[X, Y]$ and $\Phi^*{X}$ are both $S-$reversible.
  In particular, the flow  $\phi_{Y} ^t$ of $Y$ are $S-$invariant, thus $(\phi_{Y} ^t)^*{X}$ is $S-$reversible.
\end{enumerate}
\end{lemma}

\begin{lemma} [Cauchy's inequality, \cite{Geng13}]
Let  $0<\delta<r.$ For an analytic function $f(\theta, I, z, \bar{z})$ on
$D_\rho(r,s),$
$$\|\frac{\partial f}{\partial\theta_b}\|_{s;D_\rho(r-\delta,s)}\leq \frac{c}{\delta}\|f\|_{s;D_\rho(r,s)},$$
$$\|\frac{\partial f}{\partial I_b}\|_{s;D_\rho(r,s/2)}\leq \frac{c}{s}\|f\|_{s;D_\rho(r,s)},$$
and
$$\|\frac{\partial f}{\partial z^\sigma_i}\|_{s;D_\rho(r,s/2)}\leq \frac{c}{s}\|f\|_{s;D_\rho(r,s)}\me^{\rho|i|},\,\sigma=\pm.$$
\end{lemma}

\end{document}